\documentclass[aos]{imsart}

\RequirePackage{amsthm,amsmath,amsfonts,amssymb}
\RequirePackage[numbers]{natbib}
\RequirePackage{graphicx}

\usepackage{hyperref}

\startlocaldefs

\newtheorem{theorem}{Theorem}

\newtheorem{corollary}{Corollary}
\newtheorem{lemma}{Lemma}
\newtheorem{proposition}{Proposition}
\newtheorem{condition}{Condition}

\newtheorem{remark}{Remark}

\newcommand{\R}{\mathbb{R}}

\newenvironment{enumerate*}%

\endlocaldefs

\begin{document}

\begin{frontmatter}
\title{Consistent inference for diffusions from \\ low frequency measurements}

\runtitle{Inference for diffusions}

\begin{aug}
\author[A]{\fnms{Richard}~\snm{Nickl }}

\address[A]{Department of Pure Mathematics and Mathematical Statistics, University of Cambridge\footnote{I would like to thank James Norris and Gabriel Paternain for helpful discussions, three anonymous referees and the associate editor for their critical remarks, and Matteo Giordano for generating Fig.s 1-2.};~ nickl@maths.cam.ac.uk}


\end{aug}

\begin{abstract}
Let $(X_t)$ be a reflected diffusion process in a bounded convex domain in $\R^d$, solving the stochastic differential equation $$dX_t = \nabla f(X_t) dt + \sqrt{2f (X_t)} dW_t, ~t \ge 0,$$ with $W_t$ a $d$-dimensional Brownian motion. The data $X_0, X_D, \dots, X_{ND}$ consist of discrete measurements and the time interval $D$ between consecutive observations is fixed so that one cannot `zoom' into the observed path of the process. The goal is to infer the diffusivity $f$ and the associated transition operator $P_{t,f}$. We prove injectivity theorems and stability inequalities for the maps $f \mapsto P_{t,f} \mapsto P_{D,f}, t<D$. Using these estimates we establish the statistical consistency of a class of Bayesian algorithms based on Gaussian process priors for the infinite-dimensional parameter $f$, and show optimality of some of the convergence rates obtained. We discuss an underlying relationship between the degree of ill-posedness of this inverse problem and the `hot spots' conjecture from spectral geometry.  
\end{abstract}



\end{frontmatter}

\setcounter{tocdepth}{2}
\tableofcontents

\section{Introduction}

Diffusion describes a random process for the evolution over time of phenomena such as heat flow, electric conductance, chemical reactions, or molecular dynamics, to name just a few examples. The density of a diffusing substance in an insulated medium, say a bounded convex subset $\mathcal O$ of $\R^d$, $d \ge 1$, is described by the solutions $u$ to the parabolic partial differential equation (PDE) known as the heat equation, $\partial u/ \partial t= \mathcal L_{f,U}u$, with a divergence form elliptic second order differential operator $$\mathcal L_{f,U} = \frac{1}{m}\nabla \cdot (m f \nabla), ~~~m\equiv m_U \propto e^{-U},$$ and equipped with Neumann boundary conditions. Here $f :\mathcal O \to [f_{min},\infty), f_{min}>0,$ is a positive scalar `diffusivity' function and $U: \mathcal O \to \R$ is a `force' potential inducing a Gibbs measure $\mu = \mu_U$ with (Lebesgue-) probability density $m_U$. If $W_t$ is a $d$-dimensional Brownian motion then the corresponding `microscopic' statistical model for a diffusing particle is provided by solutions $(X_t)$ to the stochastic differential equation (SDE)
\begin{equation} \label{diffus0}
dX_t = \nabla f(X_t) dt + f(X_t) \nabla U(X_t)dt + \sqrt{2f (X_t)} dW_t + \nu(X_t)dL_t, ~~t \ge 0,
\end{equation}
started at  $X_0 =x \in \mathcal O$. The process is \textit{reflected} when hitting the boundary $\partial \mathcal O$ of its state space: $L_t$ is a `local time' process acting only when $X_t \in \partial \mathcal O$ and $\nu(x)$ is the (inward) pointing normal vector at $x \in \partial\mathcal O$. When $f, \nabla f, \nabla U$ are Lipschitz maps on $\mathcal O$, a continuous time Markov process $(X_t: t \ge 0)$ giving a unique pathwise solution to (\ref{diffus0}) exists \cite{T79}. 
\begin{figure}\label{figu}
\includegraphics[width=0.46 \textwidth]{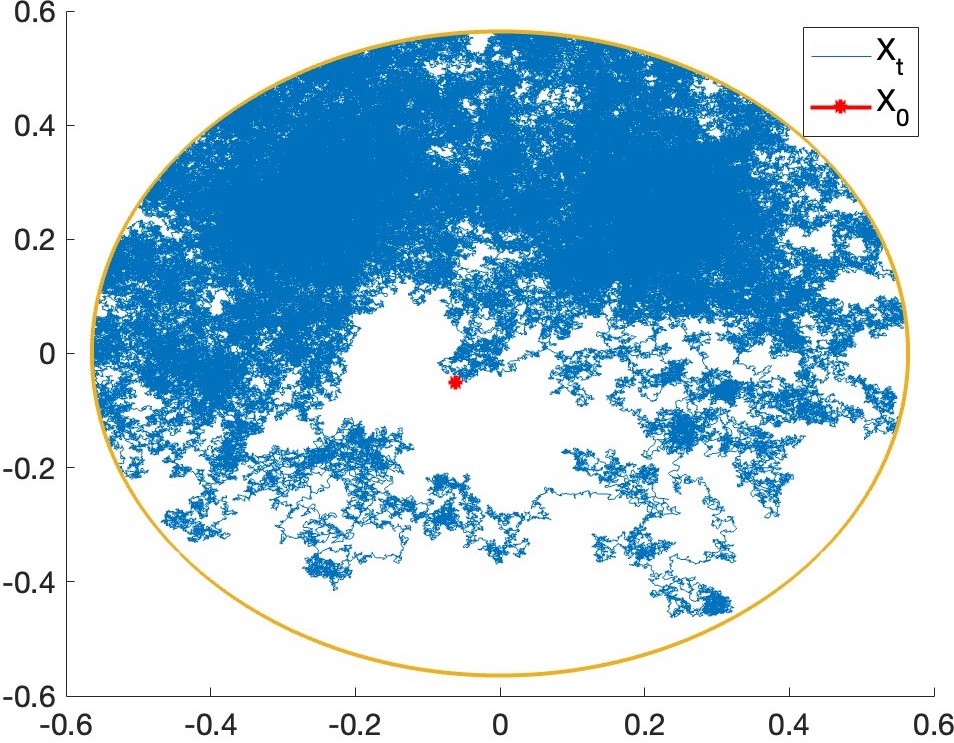}~~~ \includegraphics[width=0.46 \textwidth]{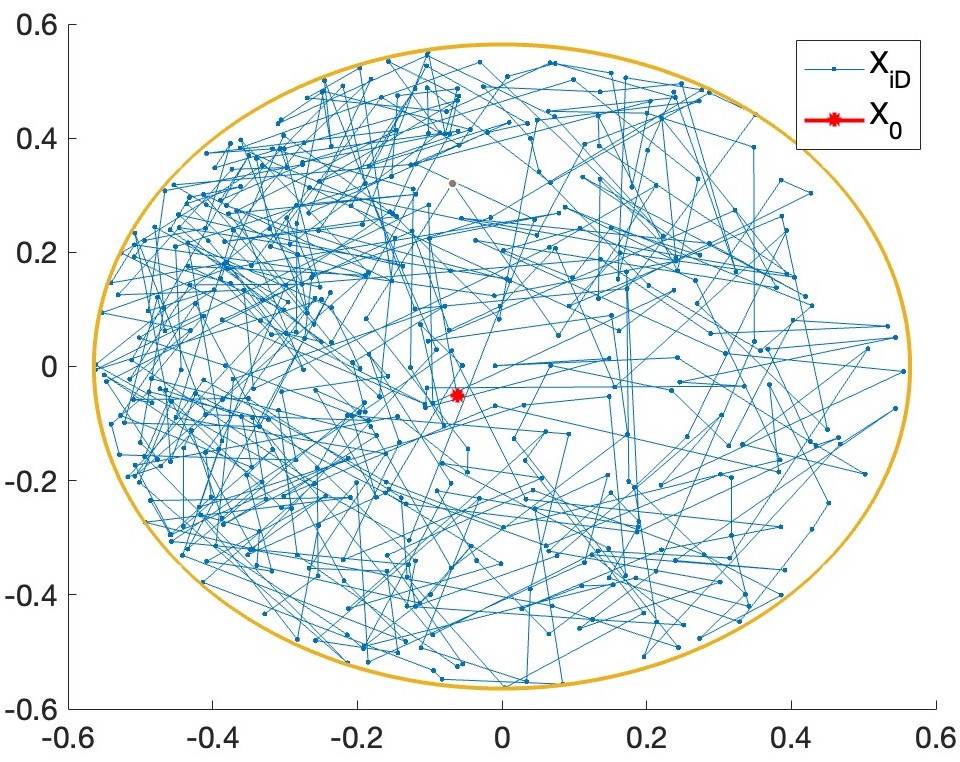}
\caption{Left: a reflected diffusion path $(X_t: 0 \le t \le T)$ initialised at $X_0$ and ran until time $T=5$. Right: $N=500$ discrete observations $(X_{iD})_{i=0}^N$ at sampling frequency $D=0.05$ ($T=25$). The diffusivity $f$ is given in Fig.~\ref{fig:estimators}.}
\end{figure}

\vspace{2pt plus 1pt minus 1pt}

Real world observations of diffusion are necessarily discrete and often subject to a lower bound on the time that elapses between consecutive measurements. We denote this `observation distance' by $D>0$ and assume for simplicity that it is the same at each measurement. The data is $X_0, X_D, \dots, X_{N D}$ for some $N \in \mathbb N$, that is, we are tracking the trajectory of a given particle along discrete points in time, see Fig.~\ref{figu}. In practice one may be observing several independent particles which essentially corresponds to (linearly) augmenting sample size $N$ -- we consider the one-particle model without loss of generality. We investigate the possibility to infer $f, U$ and the transition operator $P_{t,f, U}$ of the Markov process $(X_t)$ both at $t=D$ and at `unobserved' times $t>0$ by a statistical algorithm, that is, by a computable function of $(X_{iD}:i=1, \dots, N)$. We are interested in the scenario where $D>0$ is fixed (but known) as sample size $N \to \infty$. This is often the most appropriate observational model: for instance the speed at which particles or molecules transverse the medium $\mathcal O$ may be much faster than the frequency at which images can be taken. Following \cite{GHR04} we refer to this as the `low measurement frequency' scenario. See \cite{HDTD22, Hetal23} or also Ch.~4 in \cite{MP15} for such situations in the biological sciences, or \cite{MH12, RC15, LSZ15} in the context of data assimilation problems.

\vspace{2pt plus 1pt minus 1pt}

The invariant `equilibrium' distribution of the Markov process (\ref{diffus0}) is well known (\cite{BGL14}, Sec.1.11.3) to equal $\mu_U$ and hence identifies the potential $U$ via its probability density $m_U$. The infinite-dimensional parameter $\mu_U$ (and thus $U$) can then be estimated from a discrete sample by standard linear density estimators $\hat \mu$ that smooth the empirical measure of the $X_{iD}$'s near any point $x \in \mathcal O$ (cf.~\cite{GHR04} or also, with continuous data, \cite{DR07, S18, GR22}). Using exponentially fast mixing of ergodic averages of the Markov process towards their $\mu_U$-expectations (e.g., via \cite{P15} combined with Thm 4.9.3 and Sec.1.11.3 in \cite{BGL14}, or also with \cite{CW95}) one can then obtain excellent statistical guarantees for $\|\hat \mu - \mu_U\|$ in relevant norms $\|\cdot\|$ (e.g., as after (30) in \cite{NS17}). But the invariant measure $\mu_U$ contains \textit{no information} about the diffusivity $f$ in eq.~(\ref{diffus0}), and in a `low frequency' measurement scheme, standard statistics of the data such as the quadratic variation (`mean square displacements') of the process provide no valid inference on $f$ either (not even along the observed path). We conclude not only that recovering $f$ is a much harder problem than estimation of $U$, but also that the problems essentially decouple and can be treated separately. Therefore, to simplify the exposition of our main contributions we henceforth assume that $U=1$ in (\ref{diffus0}) and consider the model
\begin{equation} \label{diffus}
dX_t = \nabla f(X_t) dt  + \sqrt{2f (X_t)} dW_t + \nu(X_t)dL_t, ~t \ge 0,
\end{equation}
started uniformly at random $X_0 \sim Unif(\mathcal O)$. We denote by $\mathbb P_f$ the resulting probability law of $(X_t: t \ge 0)$ (in path space). Our statistical results could be generalised to the case of unknown $U$ in (\ref{diffus0}) as we discuss in Remark \ref{ohbaby} below.

\vspace{2pt plus 1pt minus 1pt}

The problem to determine diffusivity parameters $f$ from data has a long history in mathematical inverse problems -- we mention here \cite{Calderon1980, KV84, Sylvester1987, Nachman1988, Uhlmann2009, AN19} in the context of the \textit{Calder\'on problem} as well as \cite{R81, EHN96, S10, KNS08, BCDPW17, GN20, N23} in the context of \textit{Darcy's flow problem}, and the many references therein. All these settings consider a simplified observational model where one is given a `steady state' measurement of diffusion, returning the (typically `noisy') solution of a \textit{time-independent elliptic} PDE. The potential inferential barrier arising with low frequency measurements disappears in the reduction from a time evolution equation to the elliptic PDE and hence does not inform the statistical setting investigated here.

\vspace{2pt plus 1pt minus 1pt}

As the invariant measure $\mu$ is identical for all $f$, the information contained in low frequency discrete data from (\ref{diffus}) is encoded in the transition operator $P_{D, f}$ of the underlying Markov process $(X_t)$. Little is known about how to conduct statistically valid inference in this setting, with notable exceptions being the one-dimensional case $d=1$ studied in \cite{GHR04, NS17}. We also mention the consistency results \cite{vdMvZ13, GS14} as well as \cite{LP21} for Markovian transition operators, but these do not concern the conductivities $f$ themselves.  A first question is whether the task of identifying $f$ from $P_{D,f}$ for \textit{fixed} observation distance $D>0$ is even well-posed, that is, whether the (non-linear) map $f \mapsto P_{D,f}$ is \textit{injective}. The answer to this question is positive at least if $f$ is prescribed near $\partial \mathcal O$. Denote by $L^2(\mathcal O)$ the Hilbert space of square Lebesgue integrable functions on $\mathcal O$.

\begin{theorem}\label{t0}
Suppose positive diffusion coefficients $f_1, f_2 \in C^2(\mathcal O)$ are bounded away from zero on $\mathcal O$ and such that $f_1=f_2$ near $\partial \mathcal O$. Then if $P_{D,f_1}= P_{D,f_2}$ coincide as bounded linear operators on $L^2(\mathcal O)$ for some $D>0$, we must have $f_1=f_2$ on $\mathcal O$.
\end{theorem}

See Theorem \ref{logstabthm} for details. That $f$ should be known near $\partial \mathcal O$ can be explained by the fact that the reflection (which is independent of $f$) dominates the local dynamics near $\partial \mathcal O$.

\vspace{2pt plus 1pt minus 1pt}

Statistical algorithms are often motivated by `population version' identification equations for unknown parameters, as in the one-dimensional case $d=1$ considered in \cite{GHR04, NS17}, who use ordinary differential equation (ODE) techniques to derive identities for $f$ in terms of the first eigenfunction of the transition operator $P_{D,f}$. This approach appears of limited use in the present multi-dimensional context $d>1$. Instead we shall maintain $\{P_{D,f}:f \in \mathcal F\}$ as our statistical model for natural choices of parameter spaces $\mathcal F \subset L^2(\mathcal O)$ of sufficiently smooth, positive, functions. This makes available the algorithmic toolbox of Bayesian statistics in infinite-dimensional parameter spaces which does not require any identification equations or inversion formulae. Instead one employs a Gaussian process prior $\Pi$  for the function-valued parameter $f$, see  \cite{vdVvZ08, S10, GV17, N23}, and updates according to Bayes' rule: if $p_{D,f}$ are the transition densities of $P_{D,f}$ (fundamental solutions), the posterior distribution is 
$$\Pi(B|X_0, X_D, \dots, X_{ND}) = \frac{\int_B \prod_{i=1}^N p_{D,f}(X_{(i-1)D}, X_{iD}) d\Pi(f)}{\int_\mathcal F \prod_{i=1}^N p_{D,f}(X_{(i-1)D}, X_{iD}) d\Pi(f)},~~~B \text{ measurable}.$$
As the `forward map' $ f \mapsto p_{D,f}$ can be evaluated by numerical PDE techniques for parabolic equations, one can leverage ideas from \cite{CRSW13} (see also \cite{GKNSSS15, CLM16, BGLFS17}) to propose computationally feasible MCMC methodology that draws approximate samples from $\Pi(\cdot|X_0, X_D, \dots, X_{ND})$, and the resulting ergodic averages approximate the posterior mean vector, which in turn gives an estimated output for $f$. See Section \ref{baydiff}, specifically Remark \ref{compu}, for details. 

\vspace{2pt plus 1pt minus 1pt}

Recent progress in Bayesian theory for non-linear inverse problems \cite{N20, MNP21, NW20, NT23}, \cite{N23} has clarified that such Bayesian methods can solve non-linear problems without `inversion formulae' as long as appropriate stability estimates for the forward map, here $f \mapsto P_{D,f}$, are available. Following this strategy we prove here a first statistical consistency result in multi-dimensional diffusion models with such `low frequency' measurements. 

\begin{theorem}\label{showoff}
Let $D>0$ and consider data $X_0, X_D, \dots, X_{N D}$ generated from the diffusion (\ref{diffus}) in a bounded smooth convex domain $\mathcal O$. Assume the ground truth $f_0 > 1/4$ is sufficiently regular in a Sobolev sense and equals $1/2$ near $\partial \mathcal O$. Assign an appropriate Gaussian process prior $\Pi$ to $(\theta(x): x \in \mathcal O)$, form $f_\theta= (1+e^\theta)/4$, and consider the random field $$(\bar f_N \equiv f_{\bar \theta_N}(x): x \in \mathcal O),~~\bar \theta_N=E^\Pi[\theta|X_0, X_D, \dots, X_{ND}],$$ arising from the posterior mean function. Then the posterior inference for the transition operators $P_{t,f_0}, t>0,$ as well as for $f_0$ is consistent, that is, as $N \to \infty $ and in $\mathbb P_{f_0}$-probability,
$$\|P_{t,\bar f_N} - P_{t,f_0}\|_{L^2 \to L^2} \to 0,~~\text{ and }~~\|\bar f_N - f_0\|_{L^2} \to 0,$$ where $\|\cdot\|_{L^2 \to L^2}$ denotes the operator norm on $L^2=L^2(\mathcal O)$.
\end{theorem}

See Theorems \ref{maint} and \ref{main} for details. Next to the stability estimates underlying Theorem \ref{t0}, a main ingredient of our proofs is an estimate (Theorem \ref{lfsball}) on the `information' (Kullback-Leibler) distance of the underlying statistical experiment in terms of a negative Sobolev norm on $\mathcal F$. This result is of independent interest and also sharp (in view of Theorems \ref{opmx}, \ref{opmxlb}).

\vspace{2pt plus 1pt minus 1pt}

Our proofs provide a \textit{rate} of convergence in the last limits, and the rate obtained for $P_{t,f}$ cannot be improved (as we show) at the `observed time' $t=D$, corresponding to `prediction risk'. For the parameters $f$ and $P_{t,f}, t<D$, our inversion rates are potentially slow (i.e., only inverse logarithmic in $N$). The question of optimal recovery in these non-linear inverse problems is delicate as they (implicitly or explicitly) involve solving a `backward heat equation' from knowledge of $P_{D,f}$ alone. We shed some light on the issue and exhibit infinite-dimensional parameter spaces of $f$'s where faster than logarithmic rates (algebraic in $1/N$) can be obtained. These are based on certain spectral `symmetry' hypotheses on the domain $\mathcal O$ and on the diffusion process. For  $d=1$ these hypotheses are always satisfied and our theory thus recovers the one-dimensional results from \cite{GHR04, NS17} as a special case (but with novel proofs based on PDE theory). In multi-dimensions $d \ge 2$ and for $f$ in a $\|\cdot\|_\infty$-neighbourhood of the constant function, we show that the required symmetries of $\mathcal O$ can be related to the `hot spots conjecture' from spectral geometry  \cite{BB99, BW99, JN00, AB04, S20, JM20}, providing further incentives for the study of this topic. The topic of `fast' rates beyond that conjecture will be investigated in future research.

\vspace{2pt plus 1pt minus 1pt}

In principle, the Bayesian approach can be expected to give valid inferences for any measurement regime and hence should work irrespectively of whether $D \to 0$ or not. In fact, a `high frequency' regime is explicitly investigated in the recent contribution \cite{HR22} who show posterior consistency if $D \to 0$ sufficiently fast compared to $N$ (but still such that the observation horizon $N D \to \infty$). We also refer to Sec.~3.3 in \cite{GR22} for a discussion of the hypothetical case when the entire trajectory of $(X_t)$ is observed. More generally, the recent contributions \cite{S16, NR20, GR22, AWS22} to non-parametric inference for multi-dimensional diffusions (Bayesian or not) contain many further references.

\section{Main results}

We are given discrete observations $X_0, X_D, \dots, X_{N D}, N \in \mathbb N,$ of the solution $(X_t:  t \ge 0)$ of the SDE (\ref{diffus}) where $X_0 \sim Unif(\mathcal O)$, that is, the diffusion is started in its (constant) invariant distribution. If $X_0=x$ for some fixed $x$, then our proofs work as well in view of the exponentially fast mixing (\ref{mixing}) of the process towards the uniform law $\mu$, by just discarding the `burn-in phase', that is, by letting the process evolve for a while before one starts to record measurements.  We emphasise again that the time interval $D>0$ between consecutive observations remains \textit{fixed} in the $N \to \infty$ asymptotics.

The domain $\mathcal O$ supporting our diffusion process is a bounded convex open subset of $\R^d$, and to avoid technicalities we assume that the boundary of $\mathcal O$ is smooth, ensuring in particular the existence of all `reflecting' normal vectors $\nu$ at $\partial \mathcal O$. Throughout $L^2(\mathcal O)$ will denote the Hilbert space of square integrable functions for Lebesgue measure $dx$ on $\mathcal O$. We also assume (solely for notational convenience) that the volume of $\mathcal O$ is normalised to one, $vol(\mathcal O)=1$.

The physical model underlying (\ref{diffus}) describes the intensity $(u(t,x): t >0, x \in \mathcal O)$ of diffusion in an insulated medium by the equation $\partial u/\partial t = - \nabla \cdot J$ for flux $J=-f \nabla u$ (e.g., p.361f.~in \cite{TII}, and after (\ref{baseheat}) below). For smooth test functions $\phi$, let the elliptic operator $\mathcal L_f$ be given by the action
\begin{equation} \label{generator}
\mathcal L_f \phi= \nabla \cdot (f \nabla \phi) =  \nabla f \cdot \nabla \phi + f \Delta \phi = \sum_{j=1}^d \frac{\partial}{\partial x_j} \Big( f \frac{\partial }{\partial x_j} \phi \Big),
\end{equation}
where $\nabla, \nabla \cdot, \Delta$ denote the gradient, divergence and Laplace operator, respectively. Then $u$ solves the heat equation for $\mathcal L_f$ with Neumann boundary conditions $\partial u/\partial \nu=0$ on $\partial \mathcal O$. Its fundamental solutions $p_{t,f}(\cdot, \cdot): \mathcal O \times \mathcal O \to [0,\infty)$ describe the probabilities $\int_U p_{t,f}(x,y)dy$ for the position of a diffusing particle to lie in a region $U$ at time $t_0+t$ when it was at $x \in \mathcal O$ at time $t_0$. More generally the transition operator $P_{t,f}$ describes a self-adjoint action on $L^2(\mathcal O)$,  
\begin{equation} \label{traninf}
P_{t,f}(\phi) = \int_\mathcal O p_{t,f}(\cdot, y)\phi(y)dy,~~\phi \in L^2(\mathcal O).
\end{equation}
The process $(X_t: t \ge 0)$ from (\ref{diffus}) is the unique Markov random process with these transition probabilities, infinitesimal generator $\mathcal L_f$, and equilibrium (invariant) probability density $d\mu =1$ on $\mathcal O$. The generator $\mathcal L_{f}$ with Neumann boundary condition is characterised by an infinite sequence of (orthonormal) eigen-pairs $(e_{j, f}, -\lambda_{j, f}) \in L^2(\mathcal O) \times (-\infty, 0], j \ge 0,$ where $e_{0,f}$ is the constant eigenfunction corresponding to $\lambda_{0}=0$. By ellipticity the first eigenvalue satisfies the spectral gap estimate $\lambda_{1,f}>0$ (see (\ref{topevar}) below). The transition operators $P_{t,f}$ from (\ref{traninf}) can be described in this eigen-basis via the eigenvalues $\mu_{j, f} = e^{-t\lambda_{j,f}}$, and their densities $p_{t,f}$ are uniformly bounded over $\mathcal O \times \mathcal O$. These well-known facts are reviewed in Sec.~\ref{prfs}.

\vspace{2pt plus 1pt minus 1pt}

Some more notation: $C(\bar{\mathcal O})$ denotes the space of uniformly continuous functions on $\mathcal O$. The Sobolev and H\"older spaces $H^\alpha(\mathcal O), C^\alpha(\mathcal O)$ of maps defined on $\mathcal O$ are defined as all functions that have partial derivatives up to order $\alpha \in \mathbb N$ defining elements of $L^2(\mathcal O), C(\bar{\mathcal O})$, respectively, and we set $C^\infty(\mathcal O) = \cap_{\alpha>0} C^\alpha(\mathcal O)$, $C^0(\mathcal O)=C(\bar{\mathcal O})$ by convention. Attaching the subscript $c$ to any of the preceding spaces denotes the linear subspaces of such functions of \textit{compact support} within $\mathcal O$. The Sobolev sub-spaces $H^k_0$ of $H^k$ are the completions of $C^\infty_c(\mathcal O)$ for the $H^k$-norms. The symbols $\|\cdot\|_{H \to H}, \|\cdot\|_{HS}$ denote the operator and Hilbert-Schmidt (HS) norm of a linear operator on a Banach space $H$, respectively. We denote by $\|\cdot\|_\infty$ the supremum norm and by $\|\cdot\|_B$ the norm of a normed space $B$, with dual space $B^*$. Throughout, $\lesssim, \gtrsim, \simeq$ denotes inequalities (in the last case two-sided) up to fixed multiplicative constants, while $Z \sim \mu$ means that a random variable $Z$ has law $\mu$.

\subsection{Optimal recovery of the transition operator $P_{D,f}$}

Given our data, $P_{t,f}$ can be estimated directly at $t=D$ by evaluating a suitable set of basis functions of $L^2(\mathcal O)$ at the observed `transition pairs' $(X_{iD}, X_{(i+1)D})_{i =0}^{N-1}$. For instance if we take the linear span of the first $J$ eigenfunctions of the Neumann Laplacian $\mathcal L_f, f=1$, then a projection estimator for $P_{D,f}$ is described in (\ref{hatp}) below. Our first theorem establishes a bound on the convergence rate for recovery of $P_{D,f}$ in operator norm $\|\cdot\|_{L^2 \to L^2}$ if the approximating space is of sufficiently high dimension $J=J_N \to \infty$ depending on the Sobolev regularity of $f$.

\begin{theorem}\label{opmx}
Consider data $X_0, X_D, \dots, X_{N D},$ at fixed observation distance $D>0$, from the reflected diffusion model (\ref{diffus}) on a bounded convex domain $\mathcal O \subset \R^d$ with smooth boundary, started at $X_0 \sim Unif(\mathcal O)$, with $f_0 \in  C^2 \cap H^s, s>2d-1,$ such that $\inf_{x \in \mathcal O}f_0(x) \ge f_{min}>0$, $U \ge \|f_0\|_{H^s}+\|f_0\|_{C^2}$. Then the estimator $\hat P_J$ from (\ref{hatp}) with choice $J_N \simeq N^{d/(2s+2+d)}$ satisfies,
\begin{equation}\label{l2rate}
\|\hat P_J - P_{D,f_0}\|_{L^2 \to L^2} =O_{\mathbb P_{f_0}}\big(N^{-(s+1)/(2s+2+d)}\big),~~ N \to \infty,
\end{equation}
with constants $C=C(s, D,U,d, \mathcal O, f_{min})>0$ in the $O_{\mathbb P_{f_0}}$ notation. 
\end{theorem}

Our proof gives a non-asymptotic concentration inequality for the bound in (\ref{l2rate}), see Proposition \ref{opconc}. Moreover, as in Corollary \ref{opconc2} below we can deduce from (\ref{l2rate}) the convergence rate
\begin{equation}\label{alpharate}
\|\hat P_J - P_{D,f_0}\|_{H^\alpha \to H^\alpha} =O_{\mathbb P_{f_0}}\big(N^{-(s+1-\alpha)/(2s+2+d)}\big),~~0<\alpha \le s+1,
\end{equation}
for (stronger) operator norms on the $H^\alpha$ spaces. This rate is optimal in an information theoretic `minimax' sense (cf.~Ch.6 in \cite{GN16}), as we now show for the case $\alpha=2$ relevant below.

\begin{theorem}\label{opmxlb}
In the setting of Theorem \ref{opmx}, there exists a bounded convex domain $\mathcal O \subset \mathbb R^d$ with smooth boundary and a constant $c=c(s, D, U, d, f_{min})>0$ such that
\begin{equation}
\liminf_{N \to \infty} \inf_{\tilde P_N} \sup_{f: \|f\|_{H^s(\mathcal O)} \le U, f \ge f_{min}>0} \mathbb P_{f}\Big(\|\tilde P_N - P_{D,f}\|_{H^2 \to H^2} > cN^{-(s-1)/(2s+2+d)}\Big) >1/4,
\end{equation}
where the infimum extends over all estimators $\tilde P_N$ of $P_{D,f}$ (i.e., measurable functions of the $X_0, X_D, \dots, X_{N D}$ taking values in the space of bounded linear operators on $L^2$). 
\end{theorem}
The proof relies on some results from spectral geometry that require an appropriate choice of domain, in fact $\mathcal O$  is the `smoothed' hyperrectangle in (\ref{om}) below for $w \ge 2$ and $m$ large enough. The lower bound remains valid when restricting the supremum to $f$'s that are constant near $\partial \mathcal O$. The above theorems show that the minimax rate in the class of reflected diffusion processes is faster by the power of a $\log N$-factor than the minimax rate of recovery of a general Markovian transition operator in the same regularity class, cf.~Thm 2.2 in \cite{LP21}.

\subsection{Injectivity of $f \mapsto P_{t,f} \mapsto P_{D,f}$, $t<D$}

\subsubsection{Stability estimates}

We now turn to the problem of guaranteeing validity of inference on $f$, and in turn also for $P_{t,f}$ for any $t>0$. When $D \to 0$ in the asymptotics, ideas from stochastic calculus come into force and the inference problem becomes tractable either by direct techniques that identify the parameter $f$ -- see \cite{HR22} and references therein; or by steady state approximations to the diffusion equation (discussed in the introduction). 

The `low frequency' regime where $D>0$ is fixed was studied in \cite{GHR04, NS17} when $d=dim(\mathcal O)=1$. The key idea of \cite{GHR04} is to infer $f$ from a principal component analysis (PCA) of the operator $P_{D,f}$. Following their line of work when $d>1$ is not possible as they rely on explicit identification equations for $f$ based on ODE techniques (see Section 3.1 in \cite{GHR04}), and in particular on the simplicity of the first non-zero eigenvalue $\lambda_{1,f}$ of $\mathcal L_f$ -- both ideas do not extend to $d \ge 2$. Instead we follow the route via `stability estimates' used recently in work on non-linear statistical inverse problems, see \cite{MNP21}, \cite{GN20, AN19}  and also \cite{N23}. We are not aware of an explicit reference that establishes the injectivity of the `forward' map $f \mapsto P_{D,f}$ for arbitrary fixed $D>0$ (and $d \ge 2$), let alone a stability estimate. Our first result achieves this when $f$ is known near the boundary of $\mathcal O$.

\begin{theorem}\label{logstabthm}
Let $\mathcal O$ be a bounded convex domain in $\R^d, d \in \mathbb N,$ with smooth boundary.  Let $f, f_0$ be bounded from below by a constant $f_{min}>0$, suppose $f=f_0$ on $\mathcal O \setminus \mathcal O_{0}$ for some compact subset $\mathcal O_{0}$ of $\mathcal O$ and that $\|f\|_{C^2} +\|f_0\|_{C^2} \le U$ for some $U$. Then there exists a positive constant depending on $D, d,\mathcal O_0, \mathcal O, U, f_{min}$ such that
\begin{equation}\label{logstab0}
\|f-f_0\|_{L^2(\mathcal O)} \leq C  \Big(\log \frac{1}{\|P_{D,f}-P_{D,f_0}\|_{L^2 \to L^2}} \Big)^{-2/3}.
\end{equation}
In particular if $P_{D,f} = P_{D, f_0}$ co-incide as linear operators on $L^2(\mathcal O)$ for some $D>0$, we must have $f=f_0$ on $\mathcal O$.
\end{theorem}

The proof consists of a combination of the functional calculus identity $$P_{t,f} = \exp\{t\mathcal L_f\} = \exp\{t/\mathcal L_f^{-1}\}, ~t>0,$$ with injectivity estimates for the non-linear map $f \mapsto \mathcal L_{f}^{-1}(\phi)$ for appropriate $\phi$ (which have been developed earlier in related contexts, see, e.g., \cite{NvdGW20, N23} and therein).

It is of interest to improve the logarithmic  modulus of continuity in (\ref{logstab0}). We now show that at least in some regions of the parameter space of $f_0$'s this is possible. The proof strategy is substantially different from Theorem \ref{logstabthm} and instead of functional calculus relies on a spectral `pseudo-linearisation' identity for $P_{t,f}-P_{t,f_0}$ obtained from perturbation theory for parabolic PDE. This identity simplifies when testing against eigenfunctions of $P_{t,f_0}$, and allows to identify $f_0$ if a certain transport operator (related to the stability estimates for $\mathcal L_{f_0}^{-1}$) is injective. Stability of this transport operator can be reduced to a hypothesis on the eigenfunctions of $P_{t,f_0}$, which in turn can be tackled with techniques from spectral geometry.

To this end, define the first block of eigenfunctions $e_l \in H^2(\mathcal O)$ of  $-\mathcal L_{f_0}$ from (\ref{generator}) as
\begin{equation}\label{eigenblock}
E_{1,f_0, \iota}=\sum_{l: \lambda_{l, f_0} = \lambda_{1,f_0}} e_{l,f_0}\iota_l,
\end{equation}
where $\lambda_{1, f_0}$ is the first (non-zero) eigenvalue. Note that the last sum is necessarily finite and $\iota=(\iota_l)$ is any vector of scalars. The following theorem shows that under certain assumptions on $E_{1,f_0,\iota}$ to be discussed, a Lipschitz (or H\"older) stability estimate holds true.

\begin{theorem}\label{stabop}
In addition to the hypotheses of Theorem \ref{logstabthm}, assume also $\|f\|_{H^s}+\|f_0\|_{H^{s}} \le U$ for some $s>d$ and that
\begin{equation}\label{sunnyside}
\inf_{x \in \mathcal O_0}\frac{1}{2}\Delta E_{1, f_0, \iota}(x) + \mu |\nabla E_{1,f_0, \iota}(x)|_{\R^d}^2 \ge c_0>0,
\end{equation}
for some $\mu, c_0>0$ and some vector $\iota$. Then we have
\begin{equation}\label{supi}
\|f-f_0\|_{L^2} \leq \bar C \|P_{D, f}-P_{D,f_0}\|_{H^2 \to H^2}
\end{equation}
for a constant $\bar C=\bar C(U, D, \mu, c_0, \iota, \mathcal O_0, \mathcal O, f_{min}, d)$.
\end{theorem}

By standard interpolation inequalities for Sobolev norms (p.44 in \cite{LM72}) and Proposition \ref{transreg} with some $s \ge 2, k=3$, the bound (\ref{supi}) directly implies a H\"older stability estimate 
\begin{equation}\label{supiip}
\|f-f_0\|_{L^2} \lesssim \|P_{D, f}-P_{D,f_0}\|_{L^2 \to L^2}^{\gamma}
\end{equation}
for $\gamma=1/3$. Whenever $f, f_0 \in H^s$ we can let $\gamma=\gamma(s) \to 1$ as $s \to \infty$.

As we can choose $\iota$ we only need to find \textit{one} linear combination of eigenfunctions in the eigenspace for $\lambda_{1, f_0}$ that satisfies the hypothesis (\ref{sunnyside}). As multiplicities of eigenvalues reflect symmetries of $\mathcal L_{f_0}$ on $\mathcal O$, one could regard the added flexibility as a `blessing of symmetry'.

\begin{remark}[Stability for the `backward heat operator'] \normalfont \label{tdstab} 
We can write $P_{t,f}=\kappa_{t,D} (P_{D,f})$ for the operator functional $\kappa_{t,D} = \exp\{(t/D)\log (\cdot)\}$ on the spectrum $(0,1)$ of $P_{D,f}$, see the identity (\ref{trans}). For $t>D$ the map $\kappa_{t,D}$ is $C^{1+\eta}((0,1))$ for some $\eta=\eta(t,D)>0$ and one deduces from operator-norm Lipschitz properties (e.g., Lemma 3 in \cite{Kol21}) that then $\|P_{t,f}-P_{t,f_0}\|_{L^2 \to L^2} \lesssim \|P_{D,f}-P_{D,f_0}\|_{L^2 \to L^2}$. This is intuitive as the forward heat map is a smooth integral operator (the Chapman-Kolmogorov equations). In contrast in the case $t<D$, the operator functional $\kappa_{t,D}$ does not have a bounded Lipschitz constant on the spectrum. The last two theorems combined with Theorem \ref{lfsball} below (for $D=t$ there, and via the continuous imbedding $L^2 \to H^{-1}$) imply the following stability estimates for the dependence of the `backward heat operator' on $f$: Under the hypotheses of Theorem \ref{logstabthm} and assuming $\|f\|_{H^s}+\|f_0\|_{H^{s}} \le U$ for $s>d$, we can bound the $L^2(\mathcal O)$-Hilbert Schmidt norms as
\begin{equation}\label{backlog}
\|P_{t,f}-P_{t,f_0}\|_{HS} \leq C \bar {\omega} (\|P_{D,f}-P_{D,f_0}\|_{L^2 \to L^2}),~~\text{any }0<t<D,
\end{equation}
where the modulus of continuity $\bar {\omega}$ can be taken to be $\bar {\omega}(z) = \log(1/z)^{-2/3}$, and with constant $C$ now depending also on $s,t$. In light of the exponential growth of the Lipschitz constant of $\kappa_{t,D}, t<D,$ in the tail of the spectrum of $P_{D,f}$, one may think that such a logarithmic modulus of continuity is necessary. However, under the  hypothesis (\ref{sunnyside}) we can obtain a stronger H\"older modulus from our techniques. For the proof, we combine Theorem \ref{stabop} (in fact (\ref{supiip})) and Theorem \ref{lfsball} below. 
\begin{theorem}
Under the hypotheses of Theorem \ref{stabop} we have 
\begin{equation}\label{backlip}
\|P_{t,f}-P_{t,f_0}\|_{HS} \leq C' \|P_{D,f}-P_{D,f_0}\|_{L^2 \to L^2}^\gamma,~~\text{any }0<t<D,
\end{equation} where $0<\gamma<1$ is as in (\ref{supiip}) and where $C'=C(D,t,s,U, \mu, c_0, \iota, \mathcal O, \mathcal O_0, f_{min},d)$.
\end{theorem}
\end{remark}

\begin{remark}[The one-dimensional case] \label{onedee} \normalfont
In the one-dimensional setting $d=1$, Lemma 6.1 and Proposition 6.5 in \cite{GHR04} prove simplicity of $\lambda_{1,f_0}$ and the strict monotonicity in any closed subinterval $\mathcal O_0$ of $\mathcal O$ of the corresponding eigenfunction $e_{1,f_0}$ (for any $f_0 \in H^s, s>1$). This entails that the derivative $e_{1,f_0}$ cannot vanish on $\mathcal O_0$ and verifies the key hypothesis (\ref{sunnyside}) of Theorem \ref{stabop} for some $c_0>0$ and all $\mu$ large enough depending on $\|e''_{1,f_0}\|_{\infty}$ (finite if $s>2$). 
\end{remark}

We next discuss an approach to verify (\ref{sunnyside}) also in multi-dimensions $d>1$ based on the `hot spots' conjecture from spectral geometry. Ways to obtain `H\"older' stability estimates that involve eigenfunctions for multiple distinct eigenvalues (rather than just the first), can be thought of too and will be investigated in future research.

\subsubsection{Reflected diffusion and `hot spots'}

While (\ref{sunnyside}) is satisfied in dimension $d=1$ (Remark \ref{onedee}), this is less clear in higher dimensions. Indeed, if the first eigenfunction $e_{1,f}$ has a critical point in $\mathcal O_0$ with non-positive Laplacian (e.g., consider $e_{1,f}$ near $(x_1,x_2)=0$ of the form $-x_1^2 \pm x_2^2$), the condition (\ref{sunnyside}) does not hold. The hope is that eigenfunctions have special properties that exclude such situations, at least in regions $\mathcal O_0 \subset \mathcal O$ one can identify. 

Let us start with some simple examples where the condition is satisfied when $d \ge 2$. For the Laplacian ($f=const$) on the unit cube, the first eigenfunctions of $\mathcal L_1$ corresponding to $\lambda_{1,1}$ are cosines in one of the axial variables, constant otherwise, and $|\nabla e_{1,1}|_{\R^d}$ vanishes only at the respective corners of $\partial \mathcal O$. Moreover $|\Delta e_{1,1}|$ is bounded on any compact $\mathcal O_0 \subset \mathcal O$ and so we can verify (\ref{sunnyside}) for $\mu$ large enough, appropriate $\iota$, and such $\mathcal O_0$. The argument just given extends to cylindrical domains with base $\mathcal O_1$ equal to a convex domain in $\mathbb R^{d-1}$:
\begin{proposition}\label{cyllap}
Consider a cylinder $\mathcal O=\mathcal O_1 \times (0,w)$ of height $w>0$ and with convex base $\mathcal O_1 \subset \R^d$ of diameter $diam(\mathcal O_1) \le w$. Then (\ref{sunnyside}) holds for $f_0=1$, any compact $\mathcal O_0 \subset \mathcal O$, some $\iota$, and constants $\mu, c_0$ depending on $\mathcal O_0$.
\end{proposition}

Our proof shows that when $diam(\mathcal O_1)<w$, the first eigenvalue is simple and its eigenfunction satisfies (\ref{sunnyside}). When $w=diam(\mathcal O_1)$, the eigenspace of $\lambda_{1,1}$ is possibly multi-dimensional, but there always exists one eigenfunction in that eigenspace that satisfies (\ref{sunnyside}).

\vspace{2pt plus 1pt minus 1pt}

The proof of the last proposition is not difficult (see Section \ref{cylinder}) -- it draws inspiration from \cite{K85} and provides one of the few elementary examples for the validity of Rauch's \textit{hot spots conjecture}  \cite{BB99, B06} which is concerned precisely with domains $\mathcal O$ for which the gradient $\nabla e_{l,1}$ of any eigenfunction of $\Delta=\mathcal L_1$ corresponding to $\lambda_{1,1}$ has all its zeros at the boundary $\partial \mathcal O$. As the eigenfunctions are smooth in the interior of $\mathcal O$ this conjecture implies (\ref{sunnyside}) for $f=const$ and any compact $\mathcal O_0 \subset \mathcal O$ as we can then choose $\mu$ large enough depending on $\mathcal O_0, \sup_{x \in \mathcal O_0}|\Delta e_{1,1}(x)|$. The hot spots conjecture is believed to be true whenever $\mathcal O$ is convex but with the exception of cylinders has been proved only in special 2-dimensional cases so far, see \cite{JN00, AB04, JM20, S20} and references therein for positive results and \cite{BW99} who show that the conjecture may fail in non-convex domains. Next to convexity, symmetry properties of the domain $\mathcal O$ play a key role in these proofs -- in the context of Proposition \ref{cyllap} the central axis of symmetry of the cylinder `dominates the spectrum' when the base $\mathcal O_1$ is small enough, providing what is necessary to verify the conjecture in this case. The case $d=1$ from Remark \ref{onedee} can in this sense be regarded as a degenerately symmetric special case.

\vspace{2pt plus 1pt minus 1pt}

In this article we consider smooth domains but the preceding `cylinder' is not smooth near the boundary of its base. But we can `round the corners' of the cylinder without distorting the relevant spectral properties of $\mathcal L_{1}=\Delta$. For example consider $d \ge 2$ and a hyperrectangle $\mathcal O_{(w)}=(0,1)^{d-1} \times (0,w)$ for $w$ to be chosen, and define 
\begin{equation} \label{om}
\mathcal O_{m,w} = \{x \in \R^d: |x-\mathcal O_{(w)}|_{\R^d}<1/m\}, ~m \in \mathbb N.
\end{equation}
Then the $\mathcal O_{m,w}$ are bounded convex domains that have \textit{smooth} boundaries $\partial \mathcal O_{m,w}$ for all $m$, and we will show that the conclusion of Proposition \ref{cyllap} remains valid for $m$ large enough. Moreover, to lend more credence to (\ref{sunnyside}) for $f_0$ different from constant $=1$, we can extend the result to  $\mathcal L_{f_0}$ for $f_0$ in a $L^\infty$-neighbourhood of the constant function. This gives meaningful infinite-dimensional models for which the H\"older stability estimates from the previous subsection apply, and for which  `fast convergence rates' will be obtained in the next section. Incidentally they are also used to prove the lower bound in Theorem \ref{opmxlb}. For simplicity we only consider the case of \textit{simple} (first) eigenvalues in the following result.

\begin{theorem}\label{cylinderth}

A) Consider domains $\mathcal O_{m,w}$ for $w\ge 2$. Then we can choose $m$ large enough such that the Laplacian $-\Delta=-\mathcal L_1$ on $\mathcal O_{m,w}$ has a simple eigenvalue $0<\lambda_{1, 1, m}<\lambda_{2, 1, m}$ and the corresponding eigenfunction $e_{1,1,m}$ satisfies (\ref{sunnyside}) for any compact subset $\mathcal O_0$ of $\mathcal O_{(w)}$, with constant $\mu, c_0$ depending on $\mathcal O_0, d, w, m$. 

\smallskip

B) The conclusions in A) remain valid if we replace  $\mathcal L_1$ by $\mathcal L_{f_0}$ for any $f_0$ that satisfies $\|f_0\|_{H^s(\mathcal O_{m,w})} \le U, s>d,$ as well as $\|f_0-1\|_\infty<\kappa$ for some $\kappa$ small enough, with constants now depending also on $\kappa, U$.
\end{theorem}

\subsection{Bayesian inference in the diffusion model}\label{baydiff}

While we have shown injectivity of the non-linear map $f \mapsto P_{D,f}$, there is no obvious inversion formula (unless $d=1$), and so the estimate from Theorem \ref{opmx} does not obviously translate into one for $f$. The paradigm of Bayesian inversion \cite{S10} can in principle overcome such issues.  A natural Bayesian model for $f$ is obtained by placing a prior probability measure $\Pi$ on a $\sigma$-field $\mathcal S$ of some  parameter space $$\mathcal F \subset C^2(\mathcal O) \cap \big\{f: f_{min} \le \inf_{x \in \mathcal O}f(x)\big\}, ~f_{min}>0,$$ so that unique pathwise solutions to (\ref{diffus}) exist for all $f \in \mathcal F$, with transition densities $p_{D,f}$ as after (\ref{generator}). If $\mathcal B_\mathcal O$ denotes the Borel $\sigma$-field of $\mathcal O$, and if the maps $(f, x,y) \mapsto p_{D,f}(x,y)$ are jointly Borel measurable from $(\mathcal F \times \mathcal O \times \mathcal O, \mathcal S \otimes \mathcal B_\mathcal O \otimes \mathcal B_\mathcal O) \to \R$, then basic arguments (cf.~\cite{GV17} and also \cite{NS17}) show that the posterior distribution is given by
\begin{equation}\label{posterior}
\Pi(B|X_0, X_D, \dots, X_{ND}) = \frac{\int_B \prod_{i=1}^N p_{D,f}(X_{(i-1)D}, X_{iD}) d\Pi(f)}{\int_\mathcal F \prod_{i=1}^N p_{D,f}(X_{(i-1)D}, X_{iD}) d\Pi(f)},~~~B \in \mathcal S.
\end{equation}
This formula exposes the relationship of our setting to Bayesian non-linear inverse problems with PDEs \cite{S10, N23}, since the non-linear solution map $f \mapsto p_{D,f}$ of (the fundamental solution of) a parabolic PDE features in the likelihood term. Even though our measurement model is much more complex than the additive Gaussian noise models considered in \cite{S10, N23}, we can still leverage computational ideas from this literature -- see Remark \ref{compu} for details.

\vspace{2pt plus 1pt minus 1pt}

The priors $\Pi$ we consider will be of Gaussian process type. With an eye on obtaining sharp results in some cases we give a concrete construction of a prior, but the proofs below can be applied to general classes of high- or infinite-dimensional priors (commonly used in the literature \cite{GV17, MNP21, MNP21a, N23}) replacing the truncated Gaussian series in the next display. Take the first $K$ eigenfunctions $\{e_k: 0 \le k \le K\}$ of the Neumann-Laplacian $-\Delta=-\mathcal L_1$ for eigenvalues $0=\lambda_0<\lambda_1 \le \lambda_2<\dots$, and for $s\ge 0$ define a Gaussian random field
$$\theta(x) = \frac{\zeta(x)}{N^{d/(4s+4+2d)}} \Big(g_0 + \sum_{1 \le k \le K} \lambda_k^{-s/2} g_k e_k(x)\Big),~~ x \in \mathcal O,~~g_{k} \sim^{iid} N(0,1),~~ K \in \mathbb N,$$ 
where for some compact subset $\mathcal O_{0} \subset \mathcal O$, the map $\zeta \in C^\infty_c(\mathcal O)$ is a non-negative cut-off function vanishing on $\mathcal O \setminus \mathcal O_{0}$ and equal $1$ on some further compact subset $\mathcal O_{00}$ of the interior of $\mathcal O_{0}$. As in \cite{MNP21, N23}, the $N$-dependent rescaling provides extra regularisation required in the proofs -- it also allows us to remove the strong prior restrictions from \cite{NS17} in the case $d=1$.

\vspace{2pt plus 1pt minus 1pt}

For fixed $K$, the $Law(\theta)$ of $\theta$ is a probability measure supported in the space $\R^{K+1} \simeq \{\zeta g: g \in E_K\}$ where $E_K \subset C^\infty$ is the finite-dimensional linear span of the $\{e_k: 0 \le k \le K\}$. As $K \to \infty$ the $Law(\theta)$ models a $s$-smooth Gaussian random field on $\mathcal O$ that is supported in a strict subset of $\mathcal O$. The prior for the diffusivity $f \in \mathcal F = C^2 \cap \{f \ge 1/4\}$ (equipped with the trace Borel $\sigma$-algebra $\mathcal S$ of the separable Banach space $C(\bar{\mathcal O})$) is then 
\begin{equation}\label{gp}
f =f_\theta= \frac{1}{4} + \frac{e^{\theta}}{4}, ~~\Pi = Law(f)
\end{equation}
which equals $1/2$ on $\mathcal O \setminus \mathcal O_{0}$. Note that the `base case' $\theta=0$ corresponds to $f=1/2$ and hence to the case where the diffusion in (\ref{diffus}) is a standard reflected Brownian motion with generator $\mathcal L_{1/2}=\Delta/2$. The construction  can be adapted to any fixed $f_{min}>0$ replacing $1/4$.

\begin{figure}
   \centering
    \includegraphics[trim=17 17 17 10, clip, width=0.33 \textwidth]{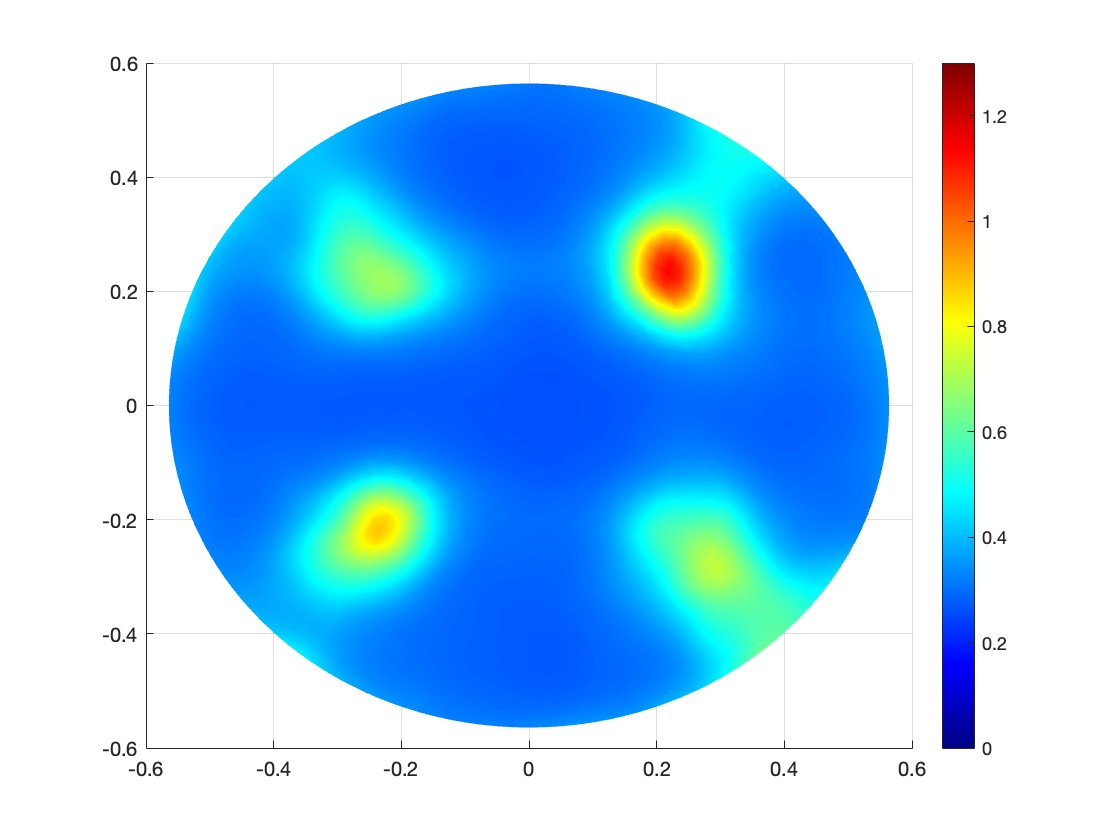}  \includegraphics[trim=20 20 20 10, clip, width=0.32 \textwidth]{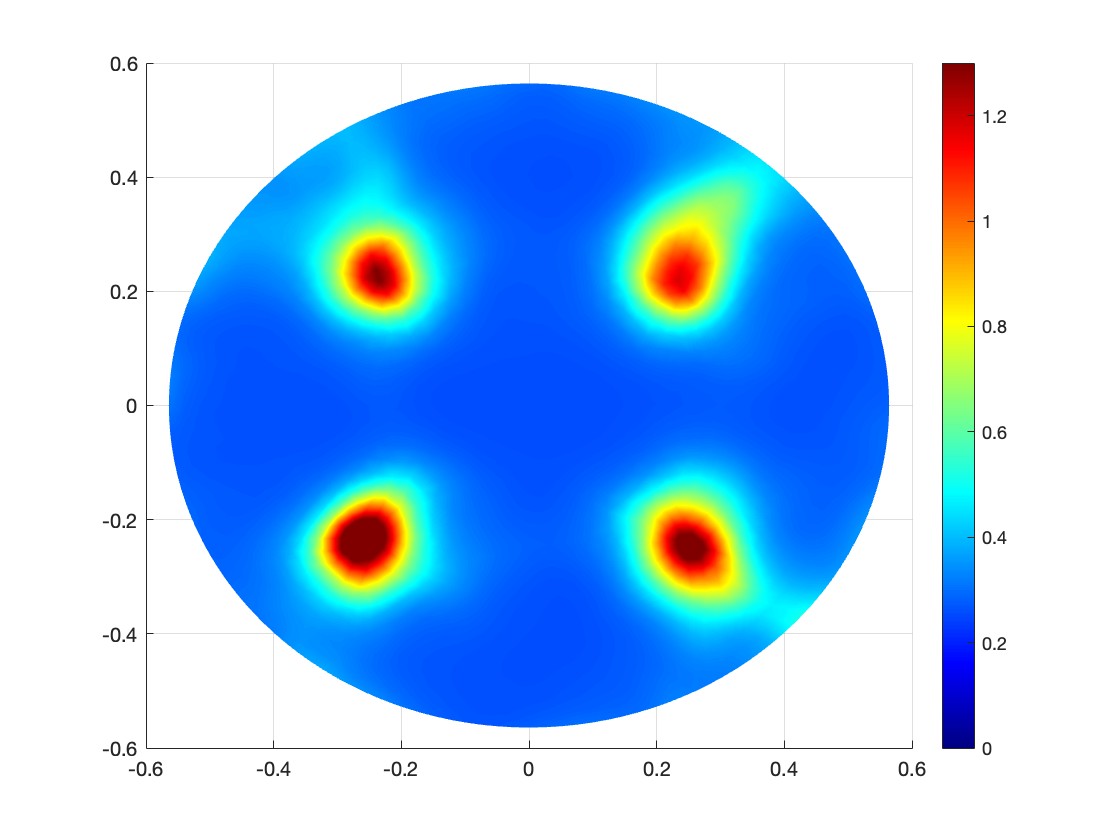}   \includegraphics[trim=20 20 20 10, clip, width=0.32 \textwidth]{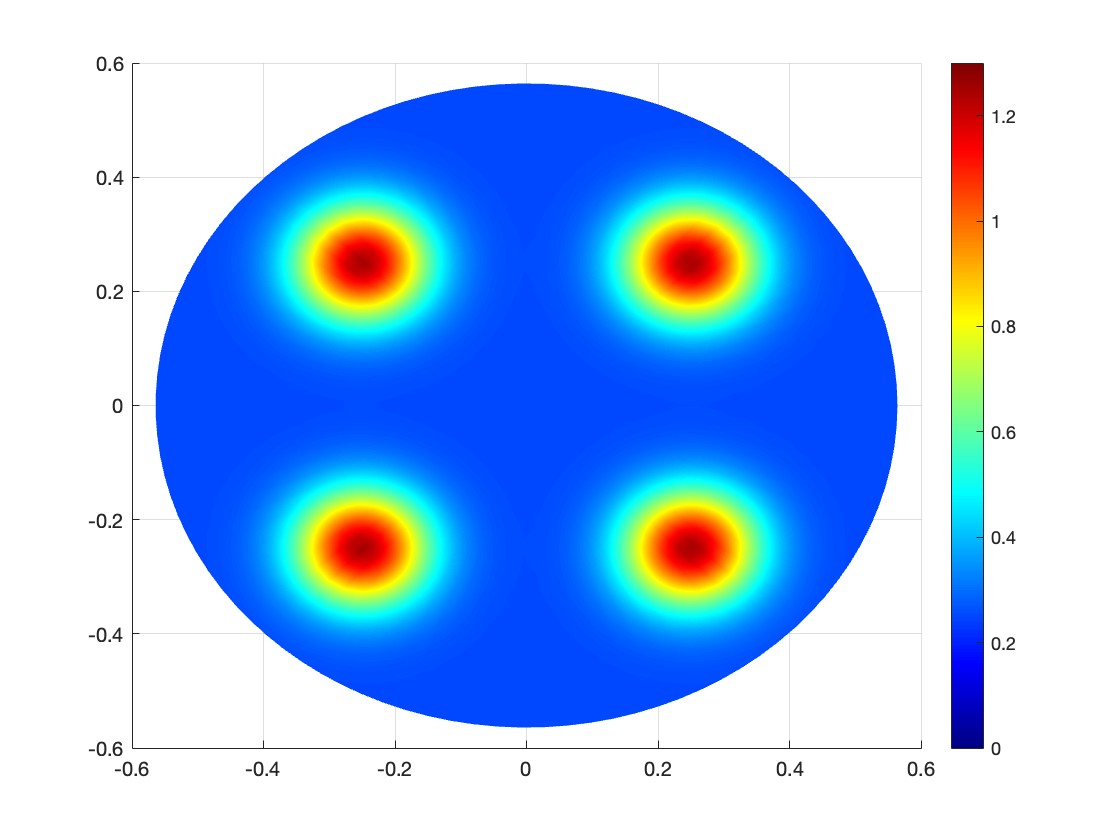}
    \caption{The posterior mean estimate $f_{\bar \theta}$ with $\bar\theta = M^{-1}\sum_{m=1}^M \vartheta_m$ after $M = 10000$ pCN iterates, for sample sizes $N=2500$ (left) and $N=25000$ (center), at sampling frequency $D=0.05$; the true field $f_0$ (right).}
    \label{fig:estimators}
\end{figure}

\begin{remark}\label{compu} \normalfont
The numerical computation of the posterior measure (\ref{posterior}) is possible via MCMC methods. For instance, since our priors are Gaussian, we can use the standard pCN proposal (see \cite{CRSW13} or Section 1.2.4 in \cite{N23}) to set up a Markov chain $(\vartheta_m)_{m=1}^M \in \R^{K+1}$ that has $\Pi(\theta|X_0, X_D, \dots, X_{ND})$ as invariant distribution. Posterior functionals $$E^\Pi [H(\theta)|X_0, X_D, \dots, X_{ND}], ~~H :\R^{K+1} \to \R^k,~~k \in \mathbb N,$$ can be approximated by ergodic averages $M^{-1}\sum_{m=1}^MH(\vartheta_m),$ see Fig.~\ref{fig:estimators} for an illustration with $H=id$. The computation of each iterate $\vartheta_m$ of this chain requires the draw of a $(K+1)$-dimensional Gaussian (from the prior) and the evaluation of the log-likelihood function $$\ell_N(\vartheta_m) \equiv \sum_{i=1}^N \log p_{D,f_{\vartheta_m}}(X_{(i-1)D}, X_{iD}).$$  In light of the representation (\ref{fundamental}), the latter can be evaluated by standard numerical methods for elliptic PDEs that compute the first few eigen-pairs $(e_{j,f_{\vartheta_m}}, \lambda_{j,f_{\vartheta_m}})$ of the differential operator $-\mathcal L_{f_{\vartheta_m}}$ with Neumann boundary conditions. Explicit error bounds for the approximation of the transition densities can be obtained from the exponential decay of the tail of the series in (\ref{fundamental}) via Corollary \ref{weyland}, and since $D>0$ is fixed in our setting. Moreover, taking limits in the pseudo-linearisation identity (\ref{repspec}) below allows to check the gradient stability condition from \cite{NW20, BN21}, which is a key to give computational guarantees for MCMC.
\end{remark}

\begin{remark}[Adding a drift] \label{ohbaby} \normalfont
The above Bayesian methodology extends to more general diffusion models (\ref{diffus0}) by proceeding as in \cite{NS17}, Sec.~2.3.2. One employs a hierarchical prior construction that first specifies a prior for the diffusivity $f$ and then models the `remaining' drift $\nabla U$ \textit{conditionally} on $f$, for instance by priors as in \cite{GR22}. One can the employ MCMC samplers for hierarchical priors, and proceed similar to \cite{vdMS17} in the `drift step'. Alternatively one can simply plug in an empirical estimate for $U$ (e.g., via an estimate $\hat \mu$ as after (30) in \cite{NS17}), avoiding hierarchical methods. When $d>1$, the case of drift vector fields in (\ref{diffus0}) that are \textit{not} in gradient form $\nabla U$ is innately more challenging as one loses self-adjointness of the infinitesimal generator. Some ideas for how to deal with such \textit{non-reversible} processes can be found in \cite{NR20, AWS22}, but for many applications, gradients of `force' potentials $U$ provide natural non-parametric models with relevant physical interpretation \cite{GR22, HDTD22, Hetal23}
\end{remark}

\begin{remark}[Different boundary conditions] \normalfont The \textit{reflected} diffusion model (\ref{diffus}) -- which corresponds to Neumann boundary conditions -- is essential to obtain a Markov process that does not `terminate' at a finite time (as would be the case for Dirichlet boundary conditions). Processes that are periodic on a $d$-dimensional cube, or that reflect along directions different from the inward normal vector at $\partial \mathcal O$, can be accommodated as well (but, at least in the latter case, introduce further tedious technicalities).
\end{remark}

\subsection{Posterior consistency theorems}

We now obtain mathematical guarantees for the inference provided by $\Pi(\cdot|X_0, X_D, \dots, X_{ND})$, following the programme of Bayesian Non-parametrics \cite{GV17} in the context of non-linear inverse problems \cite{N23}. 

\subsubsection{Posterior reconstruction of $P_{D,f}$}

We show that the Bayesian approach attains the optimal convergence rate for inference on the transition operator at the `observed' times $D$. \begin{theorem}\label{maint}
Consider discrete data $X_0, X_D, \dots, X_{N D},$ at fixed observation distance $D>0$, from the reflected diffusion model (\ref{diffus}) on a bounded convex domain $\mathcal O \subset \R^d$ with smooth boundary, started at $X_0 \sim Unif(\mathcal O)$. Assume $f_0 \in H^s, s>\max(2+d/2, 2d-1)$, satisfies $\inf_{x \in \mathcal O}f_0(x) > 1/4$ and $f_0=1/2$ on $\mathcal O \setminus \mathcal O_{00}$. Let $\Pi(\cdot|X_0, X_D, \dots, X_{ND})$ be the posterior distribution (\ref{posterior}) resulting from the  prior $\Pi$ for $f$ from (\ref{gp}) with $K \simeq N^{d/(2s+2+d)}$ and the given $s$. Then there exists $M$ depending  on $D, \mathcal O, \mathcal O_{00}, s, d$ and $U \ge \|f_0\|_{H^s}$ such that
\begin{equation}
\Pi(f: \|P_{D,f} - P_{D,f_0}\|_{L^2 \to L^2} \ge M N^{-(s+1)/(2s+2+d)}|X_0, X_D, \dots, X_{ND}) \to_{N \to \infty}^{\mathbb P_{f_0}} 0.
\end{equation}
\end{theorem}

Inspection of our proofs shows that one obtains convergence rates also in $\|\cdot\|_{H^\alpha \to H^\alpha}$ norms as in (\ref{alpharate}). When the first non-zero eigenvalue $\lambda_{1,f_0}$ of $\mathcal L_{f_0}$ is simple, the previous theorem implies consistency of the PCA provided by $P_{D,f}$. Since draws $P_{D,f}|X_0, X_D, \dots, X_{ND}$ are self-adjoint Markov transition operators, we can  extract their `principal component', or second eigenfunction, $e_{1,f}$. By the operator norm convergence of $P_{D,f}$ to $P_{D,f_0}$ the simplicity of the eigenvalue $\lambda_{1,f_0}$ eventually translates into simplicity of $\lambda_{1,f}$ with probability approaching one, and a unique $e_{1,f}$ then exists (up to choice of sign), cf.~Proposition \ref{opperture}. Using quantitative perturbation arguments (e.g., Proposition 4.2 in \cite{GHR04}) one obtains
\begin{equation}
\Pi(f: \|e_{1,f} - e_{1,f_0}\|_{L^2(\mathcal O)} \ge M N^{-(s+1)/(2s+2+d)}|X_0, X_D, \dots, X_{ND}) \to^{\mathbb P_{f_0}} 0.
\end{equation}
In dimension $d=1$, the top eigenfunction fully identifies $f$ with an explicit reconstruction formula \cite{GHR04, NS17}, but in multi-dimensions this approach is not feasible, also because $\lambda_{1,f_0}$ is not simple in general, in which case the PCA for the eigenfunction will not be consistent.

\subsubsection{Consistency and convergence rates for the non-linear inverse problem}

We now state the main statistical result of this article. 

\begin{theorem}\label{main}
Consider the setting of Theorem \ref{maint}. Then there exists a sequence $\eta_N \to 0$ such that as $N \to \infty$,
\begin{equation}
\Pi(f: \|f - f_0\|_{L^2(\mathcal O)} \ge \eta_N|X_0, X_D, \dots, X_{ND}) \to^{\mathbb P_{f_0}} 0,
\end{equation}
as well as, for any $t>0$,
\begin{equation}\label{transrate}
\Pi(f: \|P_{t,f} - P_{t,f_0}\|_{HS} \ge \eta_N|X_0, X_D, \dots, X_{ND}) \to^{\mathbb P_{f_0}} 0.
\end{equation}
Specifically we can take $\eta_N = O((\log N)^{-\delta'})$ for some $\delta'>0$. Moreover, if in addition (\ref{sunnyside}) holds for $f_0$, then we can take $\eta_N = O(N^{-(s-1)/(2s+2+d)})$. 
\end{theorem}

When $t \ge D$ we could obtain directly the convergence rate $\eta_N=N^{-(s+1)/(2s+2+d)}$ for operator norms $\|P_{t,f} - P_{t,f_0}\|_{L^2\to L^2}$ from Theorem \ref{maint} and the argument sketched at the beginning of Remark \ref{tdstab}.  But for $t<D$ we are solving a genuine inverse problem. Note further that the $HS$-norms equivalently bound the $L^2(\mathcal O \times \mathcal O, dx \otimes dx)$ norms of the difference between the transition densities $p_{t,f}-p_{t,f_0}$ from (\ref{fundamental}).

In order to obtain faster rates $\eta_N$, the hypothesis (\ref{sunnyside}) needs to hold only at the ground truth $f_0$ and not throughout the parameter space of prior diffusivities $f$. Next to the one-dimensional case discussed in Remark \ref{onedee}, Theorem \ref{cylinderth} describes an infinite-dimensional class of $f_0$'s for which such faster rates can indeed be attained also when $d\ge 2$.

Using uniform integrability type arguments as in \cite{MNP21, N23}, a similar convergence rate can be proved for the posterior mean vector $\bar \theta = E^\Pi[\theta|X_0, X_D, \dots, X_{ND}]$ and the induced conductivity $f_{\bar \theta}$ and transition operators $P_{t, f_{\bar \theta}}$, yielding Theorem \ref{showoff}. See Subsection \ref{show}.

\section{Proofs}\label{prfs}

\subsection{Analytical background: reflected diffusions and their generators}

\subsubsection{Divergence form operators}

Let $\mathcal O$ be a bounded convex domain in $\R^d$ with smooth boundary and such that $vol(\mathcal O)=1$. Consider the divergence form elliptic operator $\mathcal L_f \phi =  \nabla \cdot (f \nabla \phi)$ from (\ref{generator}). The Sobolev space $H^1(\mathcal O)$ can be endowed both with the usual norm $\|\phi\|_{H^1}=\|\phi \|_{L^2} + \|\nabla \phi\|_{L^2}$ or with the equivalent norm $\|\phi\|_{H^{1}_f} :=\|\phi \|_{L^2} + \|\sqrt f \nabla \phi\|_{L^2}$ with equivalence constants depending only on $f_{min}, \|f\|_\infty$. Moreover the elements of $H^1$ satisfying zero Neumann-boundary conditions (in the usual trace sense) are defined as $$H^1_\nu(\mathcal O) :=  \Big\{ \phi \in H^1(\mathcal O), \frac{\partial \phi}{\partial \nu} = 0 \text{ on } \partial \mathcal O  \Big \},$$ with $\nu$ the unit normal vector. By the divergence theorem (e.g., p.143 in \cite{TI})
\begin{equation}\label{greenid}
\langle \mathcal L_f \phi_1, \phi_2 \rangle_{L^2} = - \langle  f\nabla \phi_1, \nabla \phi_2 \rangle_{L^2} = \langle  \phi_1, \mathcal L_f \phi_2 \rangle_{L^2},~~\forall \phi_i \in H^1_\nu(\mathcal O),
\end{equation}
so $\mathcal L_f$ is self-adjoint for the $L^2$-inner product on $H^1_\nu$. This operator can be closed to give an operator $E_f$ on the domain $H^1(\mathcal O)$ that coincides with $-\mathcal L_f$ on $H^1_\nu$
(\cite{D95}, Theorem 7.2.1), and which corresponds to the bi-linear symmetric (Dirichlet) form
\begin{equation} \label{diric}
\mathcal E_f(\phi_1,\phi_2) = \langle  \sqrt f\nabla \phi_1, \sqrt f \nabla \phi_2 \rangle_{L^2},~~ \phi_i \in H^1(\mathcal O),
\end{equation}
which in turn defines a Markov process $(X_t: t \ge 0)$ arising from a semi-group with infinitesimal generator $\mathcal L_f$ and $d\mu(x) = dx$ as invariant probability measure. An application of Ito's formula shows that this Markov process describes solutions of the SDE (\ref{diffus}) with `reflection of the process at the boundary' provided by the (inward) normal vector $\nu$ and the `local time' process $L_t$ that is non-zero only when $X_t \in \partial \mathcal O$. Details can be found in \cite{T79}, \cite{B11} (ch.~37, 38), \cite{B98} (Sec.~I.12. and p.52) and also \cite{BGL14}.

\subsubsection{Spectral resolution of the generator}\label{specsec}

We recall here some standard facts on the spectral theory of the generator $\mathcal L_f$ with Neumann boundary conditions. The arguments follow closely the treatment of the standard Laplacian $f=1$ on p.403 in \cite{TI} (see also Ch.7.2 in \cite{D95}), and extend straightforwardly to $\mathcal L_f$ as long as $0< f_{min} \le f \le \|f\|_\infty \le U <\infty$.

Denote by $E_f$ the operator mapping $H^1$ into $L^2$ defined before (\ref{diric}).  By (\ref{diric}) the linear operator $id + E_f$ satisfies
\begin{equation}\label{h1est}
\langle (id + E_f) \phi, \phi \rangle_{L^2} = \|\phi\|_{L^2}^2 + \|\sqrt f \nabla \phi \|_{L^2}^2  = \|\phi\|_{H^{1}_f}^2 \simeq  \|\phi\|^2_{H^1},~~ \phi \in H^1,
\end{equation}
from which one deduces that the linear operator $id + E_f$ defines a bijection between $H^1$ and $(H^1)^*$ with operator norms depending only on $f_{min}, U$. If we restrict its inverse $T_{1,f}$ to the Hilbert space $L^2(\mathcal O)$ then it defines a self-adjoint operator which is also compact as it maps $L^2$ into $H^1$ which embeds compactly into $L^2$. By the spectral theorem there exist $\langle \cdot, \cdot \rangle_{L^2}$-orthonormal eigenfunctions $e_0=1$ and $e_{1,f}, \dots, e_{j,f}, \dots, \in H^1_\nu \cap L^2_0$ corresponding to eigenvalues $\lambda_0 = 0 \le  \lambda_{1,f}, \dots, \lambda_{j,f} \uparrow \infty$ such that $$\mathcal L_f e_{j,f} = - \lambda_{j,f} e_{j,f},~~j \in \mathbb N \cup \{0\}.$$  We denote by $$\mathcal L_f^{-1} =-\sum_{j \ge 1} \lambda_{j,f}^{-1} e_{j,f} \langle e_{j,f}, \cdot \rangle_{L^2}$$ the corresponding inverse operator acting on the Hilbert space $$L^2_0 :=L^2 \cap \left\{ \phi: \int_\mathcal O \phi(x)dx = \langle \phi, e_0 \rangle_{L^2} = 0\right\},$$ for which the $\{e_j: j \ge 1\}$ form an orthonormal basis. Clearly $L^2 = L^2_0 \oplus \{constants\}$.

We next record the following `uniform in $f$' spectral gap estimate: the first (nontrivial) eigenvalue $\lambda_{1,f}$ has variational characterisation (see Sec.~4.5 in \cite{D95})
\begin{align}\label{topevar}
\lambda_{1,f} & = - \sup_{u \in H^1_\nu : \langle u, 1 \rangle_{L^2} =0}\frac{\langle \mathcal L_f u, u \rangle_{L^2} }{\|u\|_{L^2}^2}  = \inf_{u \in H^1_\nu : \langle u, 1 \rangle_{L^2}=0}\frac{\langle f \nabla u, \nabla u \rangle_{L^2} }{\|u\|_{L^2}^2} \ge \frac{f_{min}}{p_\mathcal O}>0
\end{align}
where we have used the Poincar\'e-inequality (Theorem 1 on p.292 in \cite{E10}): $\|u\|_{L^2}^2 \le p_\mathcal O \|\nabla u\|_{L^2}^2$ for $u \in L^2_0$ and Poincar\'e constant $p_\mathcal O>0$ depending only on $\mathcal O$. For subsequent eigenvalues we know that they can have at most finite multiplicities (e.g., Theorem 4.2.2 in \cite{D95}), and in fact that they obey Weyl's law (e.g., using p.111 in \cite{TII}),
\begin{equation}\label{langweylig}
\lambda_{j,1} \simeq j^{2/d} \text{ as } j \to \infty.
\end{equation}
The preceding asymptotics hold initially for the standard Laplacian ($f=1$), with the constants involved depending only on $vol(\mathcal O),d$. By the variational characterisation of the $\lambda_j$'s (Sec.~4.5 in \cite{D95}) and since $$\frac{\langle f \nabla u, \nabla u \rangle_{L^2} }{\|u\|_{L^2}^2} \simeq \frac{\langle \nabla u, \nabla u \rangle_{L^2} }{\|u\|_{L^2}^2},~~f_{min} \le f \le \|f\|_\infty,$$ holds for the quadratic form featuring in (\ref{topevar}), the $\lambda_{j,f}$ corresponding to conductivities $f$ differ by at most a fixed constant that depends only on $f_{min}, \|f\|_\infty$.

Taking the eigenpairs $(e_{j,f}, \lambda_{j,f})$ of $\mathcal L_f$ one can define Hilbert spaces
\begin{equation}\label{sobnorm}
\bar H^k_f(\mathcal O) = \left\{\phi  \in L^2_0(\mathcal O) : \sum_{j \ge 1} \lambda_{j,f}^{k}\langle \phi, e_{j,f} \rangle_{L^2}^2 \equiv \|\phi\|^2_{\bar H^{k}} < \infty  \right\},~~ k \in \mathbb N.
\end{equation}
Any $\phi \in L^2_0$ can be written as $\sum_{j \ge 1} e_{j,f} \langle \phi, e_{j,f} \rangle_{L^2}$ and hence $\bar H^0_f=L^2_0$ by Parseval's identity. The following proposition (proved in Section \ref{aux}) summarises some basic properties.

\begin{proposition}\label{sobald}
Let $\mathcal O$ be a bounded convex domain in $\R^d$ with smooth boundary and let $f \in C^1(\mathcal O)$ be s.t.~$\inf_{x \in \mathcal O}f(x) \ge f_{min}>0$. Then $\bar H^1_f(\mathcal O) = H^1(\mathcal O) \cap L_0^2$ and 
\begin{equation} \label{pain}
\bar H^2_f = H^2 \cap H^1_\nu \cap L^2_0 = \big\{h \in L^2_0: \mathcal L_f h \in L^2_0, (\partial h/\partial \nu) =0 \text{ on } \partial \mathcal O\big\}.
\end{equation}
If we assume in addition that for some integer $k \ge 2$ either A)  $\|f\|_{C^{k-1}} \le U$ or B) $\|f\|_{H^s} \le U$ for some $s >d$ s.t.~$k \le s+1$, then we have $$\bar H^k_f(\mathcal O) \subset H^k(\mathcal O) \text{ and } \|\phi\|_{H^k} \simeq \|\phi\|_{\bar H^k_f}~~\text{for }\phi \in \bar H^k_f.$$ We further have the embedding $H^k_c \cap L^2_0 \subset \bar H^k_1$ and also if $H^k_c$ is replaced by $H^k_c / \R$ (modulo constants). Finally we have $\bar H^k_f = \bar H^k_{f'}$ for any pair $f,f'$ satisfying A) or B), with equivalent norms. All embedding/equivalence constants depend only on $f_{min}, U, d, k, \mathcal O$.
\end{proposition}

\begin{corollary}\label{weyland}
Under the hypotheses of Proposition \ref{sobald}B), the eigenfunctions $e_{j,f}$ corresponding to eigenvalues $\lambda_{j,f}$ of $-\mathcal L_f$ satisfy for some $C<\infty$ depending only on $\mathcal O, d, k, U, f_{min}$, 
\begin{equation}\label{efbd}
\|e_{j,f}\|_{H^k}  \lesssim \lambda_j^{k/2} \le C j^{k/d},~ j \ge 0,
\end{equation}
which whenever $k>d/2$ implies as well
\begin{equation}\label{weylsup}
\|e_{j,f}\|_\infty \lesssim j^{\tau}~~\forall \tau>1/2, ~j \ge 0.
\end{equation}
\end{corollary}
\begin{proof}
By definition (\ref{sobnorm}) and (\ref{langweylig}), the result is true for the $\bar H^k_f$-norm replacing the $H^k$-norm, and since $e_{j,f} \in \bar H^k_f$, Proposition \ref{sobald} implies (\ref{efbd}), and (\ref{weylsup}) then follows from the Sobolev imbedding.
\end{proof}

\subsection{Heat equation, transition operator, and a perturbation identity}

For fixed $T>0$ let us consider solutions $v=v_{f,\phi}: (0,T] \times \mathcal O \to \R$ in $L^2$ to the heat equation
\begin{align}\label{baseheat}
\frac{\partial}{\partial t} v - \nabla \cdot (f \nabla v) &=0 ~\text{ on }  (0,T] \times \mathcal O \\
\frac{\partial v}{\partial \nu} &= 0 ~\text{ on }  (0,T] \times \partial \mathcal O \notag \\
v(0,\cdot) &= \phi ~~\text{ on } \mathcal O, \notag
\end{align}
for any initial condition satisfying $\int_{\mathcal O} \phi =0$. The unique solution of this PDE is given by
\begin{equation}\label{trans}
v_{f, \phi}(t, \cdot) = P_{t,f}(\phi) = \sum_{j \ge 1} e^{-t \lambda_{j,f}} e_{j,f} \langle e_{j,f}, \phi \rangle_{L^2},~~t>0,~~\phi \in L^2_0(\mathcal O),
\end{equation}
which also lie in $L^2_0$. We can add any fixed constant $c$ to both the initial condition $\phi$ and solution $v$, by extending the above series to include $j=0$ for $e_0=1, \lambda_0=0$. The symmetric non-negative (Prop.~\ref{translow}) fundamental solutions of the heat equation are then
\begin{equation} \label{fundamental}
p_{t,f}(x,y) = \sum_{j \ge 0} e^{-t \lambda_{j,f}} e_{j,f}(x)  e_{j,f}(y),~~x,y \in \mathcal O.
\end{equation}
These are precisely the kernels of the transition operator $P_{t,f}$ in (\ref{traninf}) and also the transition probability densities of the Markov process $(X_t: t \ge 0)$ arising from the Dirichlet form (\ref{diric}), cf., e.g.,~Sec.1.14 in \cite{BGL14}.

\subsubsection{Heat kernel estimates}

By the bounds on eigenfunctions and eigenvalues from (\ref{langweylig}), (\ref{weylsup}), the series in (\ref{fundamental}) defining $p_{t,f}$ converge in $H^k$, and by the Sobolev imbedding with $k>d/2$ then also uniformly on $\mathcal O$.

\begin{proposition}\label{transreg}
Under the hypotheses of Proposition \ref{sobald}B),  we have for any fixed $t>0$
\begin{equation}
\sup_{x \in \mathcal O} \|p_{t,f}(x,\cdot)\|_{H^k} \le c_{ub} < \infty.
\end{equation}
where $c_{ub}=c_{ub}(k, t, f_{min}, U, \mathcal O, d)<\infty$. \end{proposition}
\begin{proof}
Using the representation (\ref{fundamental}) and Corollary \ref{weyland} we obtain
\begin{align*}
\|p_{t,f}(x, \cdot)\|_{H^k} \le \sum_{j \ge 0}  e^{-t\lambda_j} \|e_j\|_{H^k} \|e_j\|_\infty \lesssim \sum_{j \ge 0} j^{\tau + (k/d)} e^{-c t j^{2/d}} \leq c_{ub}.
\end{align*}
\end{proof}
A further key fact is that the transition densities are bounded \textit{from below} on a convex domain $\mathcal O$. See Section \ref{aux} for the proof.
 \begin{proposition} \label{translow}
Let $\mathcal O$ be a bounded convex domain with smooth boundary and suppose $f \ge f_{min}>0$ satisfies $\|f\|_{C^\alpha} \le B$ for some even integer $\alpha > (d/2)-1$. Then we have for every $t>0$ and some positive constant $c_{lb}(t, \mathcal O, d, f_{min}, B, \alpha)>0$ that 
\begin{equation}
\inf_{x,y \in \mathcal O}p_{t,f}(x,y) \ge c_{lb}.
\end{equation}
\end{proposition}

Using Proposition 6.3.4 in \cite{BGL14} and (\ref{topevar}), (\ref{LI}) (or by estimating the tail of the series in (\ref{fundamental}) and integrating the result $dx$) one also obtains geometric ergodicity of the diffusion process,
\begin{equation} \label{mixing}
\sup_{x \in \mathcal O}\|p_{t,f}(x,\cdot) - \mu\|_{TV} \le C e^{-\lambda_{1,f} t},~~\forall t \ge t_0>0,~~d\mu(x)=e_0(x)=1, ~x \in \mathcal O.
\end{equation}

\subsubsection{Perturbation and pseudo-linearisation identity}

In this subsection we consider two conductivities $\bar f, f' \ge f_{min}>0$ whose $C^2(\mathcal O)$-norms are bounded by a fixed constant $U$ and study the resulting difference of the action of the transition operators $P_{t,\bar f} - P_{t,f'}$ on the eigen-functions $(e_{j,\bar f}: j \ge 1) \subset H^2(\mathcal O)$ of $\mathcal L_{\bar f}$. We will use the factorisation of space and time variables in the identity $P_{t,\bar f}(e_{j,\bar f}) =  e^{-t\lambda_{j,\bar f}} e_{j,\bar f}$ 
which holds as well for the eigenblocks (with $\iota=(\iota_l)$ any finite sequence)
\begin{equation}\label{eblo}
E_{j, \bar f,\iota} = \sum_{l: \lambda_{l,\bar f}=\lambda_{j,\bar f}} e_{l, \bar f} \iota_l
\end{equation}
corresponding to the eigenvalue $\lambda_{j, \bar f}$, that is, we have
\begin{equation} \label{actblock}
P_{t,\bar f}(E_{j,\bar f,\iota}) =  e^{-t\lambda_{j,\bar f}} E_{j,\bar f, \iota},~~ j \ge 0.
\end{equation}
By (\ref{baseheat}), (\ref{trans}), the functions
$$v_j(\cdot, t)=P_{t,f'}(E_{j, \bar f,\iota})-P_{t,\bar f}(E_{j, \bar f,\iota}),~ t \in (0,T],~j \ge 1,$$ solve the inhomogeneous PDE
\begin{align}\label{line200}
\frac{\partial}{\partial t} v - \nabla \cdot (f' \nabla v) &=\bar G_j ~\text{ on }  (0,T] \times \mathcal O \\
\frac{\partial v}{\partial \nu} &= 0 ~\text{ on }  (0,T] \times \partial \mathcal O \notag \\
v(0,\cdot) &= 0 ~~\text{ on } \mathcal O \notag
\end{align}
where 
\begin{equation} \label{gee0}
\bar G_j(t) = -\nabla \cdot [(\bar f-f')\nabla P_{t,\bar f}(E_{j, \bar f,\iota})] = e^{-t\lambda_{j,\bar f}} G_j, \quad G_j:= -\nabla \cdot [(\bar f-f')\nabla E_{j, \bar f,\iota}],
\end{equation}
with  eigenvalues $\lambda_{j,\bar f}$ of $-\mathcal L_{\bar f}$. Standard semi-group arguments (Proposition 4.1.2 in \cite{L95}) imply that the solution $v$ of (\ref{line200}) can be represented by the `variation of constants' formula
\begin{equation}\label{inhom0}
v_j(\cdot, t) = \int_0^t e^{(t-s)\mathcal L_{f'}} \bar G_j(s)ds .
\end{equation}
For $(e_{k, f'}, \lambda_{k, f'})$ the eigen-pairs of $-\mathcal L_{f'}$ we thus arrive at 
\begin{align} \label{repspec}
P_{t,f'}(E_{j,\bar f, \iota})-P_{t,\bar f}(E_{j,\bar f, \iota})&=v_j(\cdot, t) =\sum_{k \ge 1} \int_0^{t}e^{-s\lambda_{j,\bar f}}  e^{-(t-s)\lambda_{k,f'}}  \langle e_{k,f'}, G_j \rangle_{L^2} e_{k,f'} ds \\
&\equiv \sum_k b_{k,j} \langle e_{k,f'}, G_j\rangle_{L^2} e_{k,f'} \notag,~~j \ge 1,
\end{align}
for coefficients 
\begin{equation} \label{tough0}
b_{k,j}= b_{k,j}(t)= \int_0^{t}e^{-s\lambda_{j,\bar f}}  e^{-(t-s)\lambda_{k,f'}} ds.
\end{equation}
We can regard (\ref{repspec}) as a spectral `pseudo-linearisation' identity for $P_{t,f'}-P_{t,\bar f}$, similar to analogous results employed to prove stability estimates in other inverse problems, e.g., \cite{MNP21}. It could also be the starting point to prove LAN-type expansions in our model as in \cite{W19}.

\subsection{Information distances and small ball probabilities}

For $(X_t: t \ge 0)$ the diffusion process (\ref{diffus}) with transition densities from (\ref{fundamental}), the Kullback-Leibler (KL-) divergence in our discrete measurement model with observation distance $D>0$ is defined as
\begin{equation}\label{KLpobs}
KL(f,f_0) = E_{f_0} \Big[\log \frac{p_{D, f_0}(X_0, X_D)}{p_{D, f}(X_0, X_D)} \Big],~~f, f_0 \in \mathcal F,
\end{equation}
where we regard the $p_{D, f}$ from (\ref{fundamental}) as joint probability densities on $\mathcal O \times \mathcal O$ (as $vol(\mathcal O)=1$).

\vspace{2pt plus 1pt minus 1pt}

In the following theorem $\|\cdot\|_{HS}$ denotes the HS norm for operators on the Hilbert space $L^2(\mathcal O)$ (or just $L^2_0(\mathcal O)$). Note further that $H^1_c \subset H^1_0$ implies $(H^1_0)^*=H^{-1} \subset (H^1_c)^*$ so the r.h.s. in (\ref{weakball}) can be bounded by $\|f-f_0\|^2_{H^{-1}}$ and then also by $\|f-f_0\|_{L^2}$.

\begin{theorem}\label{lfsball}
Let $f,f_0$ satisfy the conditions of Proposition \ref{sobald}B) for some $s>d$. Suppose $f=f_0$ outside of a compact subset $\mathcal O_0 \subset \mathcal O$. Then for any $D>0$ there exist positive constants $C_0, C_1$ depending on $D, \mathcal O, \mathcal O_0, s,d, U, f_{min}$ such that
\begin{equation} \label{weakball}
KL(f,f_0) \leq C_0 \|P_{D,f_0} - P_{D,f}\|^2_{HS} \leq C_1 \|f-f_0\|_{(H^1_c)^*}^2.
\end{equation}
\end{theorem}
\begin{proof}
Using Propositions \ref{transreg}, \ref{translow} (noting also $H^{s} \subset C^\alpha$ by the Sobolev imbedding) and standard inequalities from information theory (as at the beginning of the proof of Lemma 14 in \cite{NS17}, or see Appendix B in \cite{GV17}) one shows
\begin{equation}\label{HSred}
KL(f,f_0) \lesssim c(c_{ub}, c_{lb}) \|p_{D, f_0} - p_{D,f}\|^2_{L^2(\mathcal O \times \mathcal O)} \lesssim  \|P_{D,f_0} - P_{D,f}\|^2_{HS}.
\end{equation}
The HS-norm of an operator $A$ on any Hilbert space $H$ can be represented as
$\|A\|^2_{HS} = \sum_j \|Ae_j\|_{L^2}^2 $ where the $(e_j)$ are any orthonormal basis of $H$. In what follows we take the basis $(e_j)\equiv (e_{j,f})$ arising from the spectral decomposition of $\mathcal L_f$, and hence need to bound
\begin{equation} \label{target}
\sum_{j \ge 1}\|P_{D,f_0}(e_{j,f})-P_{D,f}(e_{j,f})\|_{L^2}^2
\end{equation}
where the HS-norms can be taken over the Hilbert space $L_0^2(\mathcal O, dx)$ as both operators have identical first eigenfunction $e_{0,f_0}=1=e_{0,f}$.
For each summand $P_{D,f_0}(e_{j,f})-P_{D,f}(e_{j,f})$ we apply the representation (\ref{repspec}) with $f'=f_0, \bar f =f$ and $\iota_l$ selecting the relevant $j$-th eigenfunction if there are multiplicities. We then write shorthand $$g_j =-\nabla \cdot [(f-f_0)\nabla e_{j, f}]$$ for $G_j$ from (\ref{gee0}) with these choices. We can bound the coefficients (\ref{tough0}) as
\begin{align*}
|b_{k,j}|&= e^{-t\lambda_{k,f_0}} \Big( \int_0^{t/2}e^{-s\lambda_{j,f}}  e^{s\lambda_{k,f_0}} ds + \int_{t/2}^t e^{-s\lambda_{j,f}}  e^{s\lambda_{k,f_0}} ds \Big) \\
&\le e^{-t\lambda_{k,f_0}/2} \lambda_{j,f}^{-1} +   e^{-t \lambda_{j,f}/2} \lambda_{k, f_0}^{-1},
\end{align*}
so that by Parseval's identity, for $0<t\le T$, and writing $(e_k, \lambda_k)=(e_{k,f_0}, \lambda_{k, f_0})$ for the remainder of the proof,
\begin{align*}
\|P_{t,f_0}(e_{j,f})-P_{t,f}(e_{j,f})\|_{L^2}^2 = \sum_k b^2_{k,j}\langle e_k, g_j \rangle^2  &\lesssim \sum_{k \ge 1} e^{-t\lambda_k}\lambda_{j,f}^{-2} \langle e_k, g_j \rangle^2 + \sum_{k \ge 1} e^{-t\lambda_{j,f}}\lambda_{k}^{-2} \langle e_k, g_j \rangle^2.
\end{align*}
Returning to (\ref{target}) we are thus left with bounding the double sum
\begin{equation} \label{pertubd}
\sum_{j \ge 1}\|P_{D,f_0}(e_j)-P_{D,f}(e_j)\|_{L^2}^2 \lesssim \sum_{j,k}  e^{-D\lambda_k}\lambda_{j,f}^{-2} \langle e_k, g_j \rangle^2 + \sum_{j, k} e^{-D\lambda_{j,f}}\lambda_{k}^{-2} \langle e_k, g_j \rangle^2 .
\end{equation}
By the divergence theorem $$\langle e_k, g_j \rangle_{L^2} = \langle e_k, \nabla \cdot [(f-f_0)\nabla e_{j,f}] \rangle_{L^2} = \langle e_{j,f},  \nabla \cdot [(f-f_0)\nabla e_k] \rangle$$ so by Parseval's identity and (\ref{sobnorm}) (with norm there well-defined also for negative $k$), the r.h.s.~in (\ref{pertubd}) is bounded by
\begin{equation}\label{symmetric}
\sum_{k}  e^{-D\lambda_k} \|\nabla \cdot [(f-f_0)\nabla e_k]\|_{\bar H_f^{-2}}^2 + \sum_{j} e^{-D\lambda_{j,f}} \|\nabla \cdot [(f-f_0)\nabla e_{j,f}]\|_{\bar H_{f_0}^{-2}}^2.
\end{equation}
In the next step we use the basic sequence space duality relationship $\bar H_f^{-2} = (\bar H_f^2)^*$. Moreover, noting $f=f_0$ outside of $\mathcal O_0$, we take a suitable smooth cut-off function $\zeta$ that equals one on $\mathcal O_0$ and is compactly supported in $\mathcal O$. Then we apply the divergence theorem in conjunction with Proposition \ref{sobald} to obtain
\begin{align*}
\| \nabla \cdot [(f-f_0)\nabla e_{j,f}]\|_{\bar H_{f_0}^{-2}} &  = \sup_{\|\psi\|_{\bar H_{f_0}^2}\le 1} \Big| \int_\mathcal O \psi \nabla \cdot [(f-f_0) \nabla  e_{j,f}] \Big| \\
& = \sup_{\|\psi\|_{\bar H_{f_0}^2}\le 1} \Big| \int_\mathcal O (f-f_0) \nabla  (\zeta \psi) \cdot \nabla e_{j,f} \Big| \\
&\le \sup_{\bar \psi \in H^2_c, \|\bar \psi\|_{H^2} \le c} \|\nabla e_{j,f} \cdot \nabla \bar \psi]\|_{H^1} \|f-f_0\|_{(H^1_c)^*} \\
& \lesssim \|f-f_0\|_{(H^1_c)^*} \sup_{\|\bar \psi\|_{H^2} \le c} \|\bar \psi\|_{H^2} \|e_{j,f}\|_{B^2}.
\end{align*}
with spaces $B^2$ as after (\ref{multi}). For $d \le 3$ we have $B^2=H^2$ and then $\|e_{j,f}\|_{B^2}\lesssim j^{2/d}$ in view of Corollary \ref{weyland} with $k=2 \le s+1$. For $d>3$ and $k=2+d/2+\eta \le s+1, \eta>0,$ we use the Sobolev embedding $H^k \subset C^2=B^2$ and again Corollary \ref{weyland} to bound $\|e_{j,f}\|_{C^2}$. In both cases the r.h.s.~in the last display is bounded by a constant multiple of $j^{c(d)} \|f-f_0\|_{(H^1_c)^*}$ for some constant $c(d)>0$. Inserting these bounds into the second summand in (\ref{symmetric}) and using (\ref{langweylig}), the series $$\sum_j j^{2c(d)} e^{-cDj^{2/d}}<\infty$$ is convergent (for $D>0$ fixed). The same estimate holds for $e_{j,f}, \lambda_{j,f}, \bar H^{-2}_{f_0}$ replaced by $e_k, \lambda_k, \bar H^{-2}_{f}$, summing the first term in (\ref{symmetric}) -- completing the proof of the theorem. 
\end{proof}

\subsection{Proofs of stability estimates}

\subsubsection{Proof of Theorem \ref{logstabthm}} \label{loggi}

Take $\phi \in C_c^\infty(\mathcal O)$ such that $\phi =1$ on $\mathcal O_0$ and $\int_\mathcal O \phi =0$ (as $\mathcal O_0$ is a compact subset of $\mathcal O$, such $\phi$ exists).  By the results from Section \ref{specsec}, the inhomogeneous elliptic PDE (\ref{aeneumann}) in Lemma \ref{zauberstab} below has unique solution 
\begin{equation} \label{ellip}
u_{f,\phi} = \mathcal L_{f}^{-1} \phi =  -\sum_{j=1}^{\infty} \lambda_{j,f}^{-1} e_{j,f} \langle e_{j,f}, \phi \rangle_{L^2(\mathcal O)}.
\end{equation}
In particular Proposition \ref{sobald} implies that $\phi \in \bar H^{2}_{f}$  and that $u_{f,\phi}$ is bounded in $\bar H^{4}_f \subset H^{3}$. The same arguments apply to $f_0$ replacing $f$. Now Lemma \ref{zauberstab} implies
\begin{equation}
\|f-f_0\|_{L^2(\mathcal O)}  \lesssim \|f_0\|_{C^1} \|u_{f, \phi} - u_{f_0, \phi}\|_{H^2(\mathcal O)} \leq C \|u_{f,\phi} - u_{f_0,\phi}\|_{L^2}^{1/3}
\end{equation}
for finite constant $C=C(\|u_{f,\phi}\|_{H_f^3},\|u_{f_0,\phi}\|_{H^3})\le C(U)$, where we have also used the standard interpolation for $H^2$-norms (p.44 in \cite{LM72}). We now estimate the right hand side in the last display. As $\phi \in L^2$ we have for any $J \in \mathbb N$ that
\begin{align*}
\big\|u_{f, \phi} - \sum_{j \le J} (-\lambda_{j,f}^{-1}) e_{j,f} \langle e_{j,f}, \phi \rangle_{L^2} \big\|^2_{L^2} & \le \sum_{j>J} \lambda_{j,f}^{-2}\langle e_{j,f}, \phi \rangle_{L^2}^2 \le C_{\phi,f_{min},U}  J^{-c(d)},
\end{align*}
for $c(d)=4/d$, using also (\ref{langweylig}), and similarly for $f=f_0$. By the triangle inequality
\begin{equation}\label{decompoJ}
\|u_{f, \phi} - u_{f_0, \phi}\|_{L^2} \le  \Big\|\sum_{j \le J} \lambda_{j,f}^{-1} e_{j,f} \langle e_{j,f}, \phi \rangle_{L^2} - \sum_{j \le J} \lambda_{j,f_0}^{-1} e_{j,f_0} \langle e_{j,f_0}, \phi \rangle_{L^2} \Big\|_{L^2} + 2C_{\phi,f_{min}, U}J^{-c(d)}.
\end{equation}
Let us further define `truncated' transition operators 
$$P_{D,f,J}(\phi) = \sum_{j \le J} e^{-D\lambda_{j,f}} e_{j,f} \langle e_{j,f}, \phi \rangle_{L^2},~~~\mu_{j,f} = e^{-D\lambda_{j,f}},~\phi \in L^2_0,$$
which, just as in the display above (\ref{decompoJ}) and in view of (\ref{langweylig}), satisfy the estimate
$$\|P_{D,f} - P_{D,f,J}\|_{L^2 \to L^2} \le e^{-\bar cJ^{2/d}},~~\bar c = \bar c(D,U,f_{min})>0,$$ and the same is true for $f_0$ replacing $f$. The operators $P_{D,f,J}$ are self-adjoint on $L^2_0(\mathcal O)$ and by what precedes and (\ref{langweylig}), (\ref{topevar}), the union of their spectra is contained in $$\big[\min_{f,f_0} \mu_{J,f}, \max_{f,f_0} \mu_{1,f}\big] \subset \big[e^{-c' J^{2/d}},e^{-Df_{min}/p_\mathcal O}\big],~c'=c'(D,U,f_{min})>0.$$
We can employ a cut-off function and construct smooth $\kappa_J$ compactly supported on $(e^{-c' J^{2/d}}/2,1)$ such that $$\kappa_J (z)= -\frac{D}{\log z},~ \text{ on } \Big[\min_{f,f_0} \mu_{J,f}, \max_{f,f_0} \mu_{1,f}\Big].$$ Then since $\lambda_j^{-1} = \kappa_J(e^{-D\lambda_j}) = \kappa_J (\mu_j)$ on the last interval, we can write, using the notation of functional calculus,
\begin{align*}
&\Big\|\sum_{j \le J} \lambda_{j,f}^{-1} e_{j,f} \langle e_{j,f}, \phi \rangle_{L^2} - \sum_{j \le J} \lambda_{j,f_0}^{-1} e_{j,f_0} \langle e_{j,f_0}, \phi \rangle_{L^2} \Big\|_{L^2} \\
&\lesssim \|\kappa_J(P_{D, f, J}) - \kappa_J(P_{D, f_0, J})\|_{L^2 \to L^2} \\
& \lesssim \|\kappa_J\|_{B^1_{\infty 1}(\R)} \|P_{D, f,J}- P_{D, f_0,J}\|_{L^2 \to L^2}  \lesssim e^{cJ^{2/d}} \|P_{D, f} - P_{D,f_0}\|_{L^2 \to L^2} + e^{-\bar cJ^{2/d}} 
\end{align*}
where we have used Lemma 3 in \cite{Kol21} for the self-adjoint operators $P_{D,f, J}, P_{D,f_0, J}$ on $L^2_0$ and the bound $\|\kappa_J\|_{B^1_{\infty 1}(\R)} \lesssim e^{cJ^{2/d}}$ (using results in Sec.~4.3 in \cite{GN16}).  Combining all that precedes, we obtain the overall estimate
$$\|f-f_0\|^3_{L^2(\mathcal O)} \lesssim e^{cJ^{2/d}} \|P_{D, f} - P_{D,f_0}\|_{L^2 \to L^2} + e^{-\bar c J^{2/d}} + J^{-c(d)}$$ 
where $J \in \mathbb N$ was arbitrary. Choosing $J$ such that $$J^{2/d}= \frac{1}{2c}\log \frac{1}{\|P_{D, f} - P_{D,f_0}\|_{L^2 \to L^2}} $$ (we can increase $2c$ if necessary to ensure $J \in \mathbb N$) implies for some $\delta'=\delta'(c,\bar c)>0$ that
\begin{equation*}
\|f-f_0\|^3_{L^2(\mathcal O)} \lesssim \log \Big(\frac{1}{\|P_{D, f} - P_{D,f_0}\|_{L^2 \to L^2}} \Big)^{-\delta} + \|P_{D, f} - P_{D,f_0}\|^{\delta'}_{L^2 \to L^2},~~\delta = c(d)d/2.
\end{equation*}
As the $\|f-f_0\|_{L^2} \le 2U$ are uniformly bounded, we can absorb the second term into the first after adjusting constants, so the stability estimate is proved, and the injectivity assertion of the theorem follows directly from it.

\subsubsection{Proof of Theorem \ref{stabop}}

For eigenblocks $E_{1,f_0,\iota} \in \bar H^2_{f_0}$ from (\ref{eblo}), Proposition \ref{sobald} gives $$\|E_{1,f_0,\iota}\|_{H^2} \lesssim \|E_{1,f_0,\iota}\|_{\bar H^2_{f_0}} = |\iota| \lambda_{1,f_0}<\infty,~~\text{where }|\iota|^2=\sum_l \iota_l^2.$$ Then, using the representation (\ref{repspec}) with choices $\bar f=f_0, f'=f$ and Proposition \ref{sobald},
\begin{align}\label{newstab}
\|P_{D,f} -P_{D,f_0}\|^2_{H^2 \to H^2} &\gtrsim \|P_{D,f}(E_{1,f_0,\iota})-P_{D,f_0}(E_{1,f_0,\iota})\|^2_{\bar H_f^2} \notag \\
&= \sum_k \lambda_{k,f}^2 |b_{k,1}|^2 |\langle G, e_{k,f}\rangle|^2,
\end{align}
where $G=\nabla \cdot [(f-f_0)\nabla E_{1,f_0,\iota}]$ and $b_{k,1} = \int_0^t e^{-s\lambda_{1,f_0}} e^{-(t-s)\lambda_{k, f}}ds$. We can write \begin{align*}
b_{k,1}&=e^{-t\lambda_{k,f}} \frac{e^{-t(\lambda_{1,f_0}-\lambda_{k,f})}-1}{\lambda_{k,f} - \lambda_{1,f_0}} = t \frac{e^{-t \lambda_{1,f_0}} - e^{-t\lambda_{k,f}}}{t(\lambda_{k,f} - \lambda_{1,f_0})} = t e^{\xi(\lambda_{1, f_0}, \lambda_{k,f})}
\end{align*}
for some mean values $\xi(\lambda_{1, f_0}, \lambda_{k,f})$ in the interval $[-t \lambda_{1,f_0}, -t \lambda_{k,f}]$ arising from the mean value theorem applied to the exponential map. This remains true in the degenerate case where $\lambda_{1,f_0}=\lambda_{k,f}$ as then $b_{1,k} = t e^{-t\lambda_{1, f_0}}$. 

Now recalling the distribution of the eigenvalues from (\ref{langweylig}) we see that for $k \le K$ with $K$ fixed, the last displayed exponential is bounded below by a fixed constant depending on $K,d$, while for large values of $k$, the last but one term in the last display is of order $1/\lambda_{k,f}$ for $t$ fixed. Hence we have for all $k$, and some $C=C(t, d, \mathcal O, f_{\min},U)$,
\begin{equation}
|b_{k,1}| \ge C \lambda_{k,f}^{-1}.
\end{equation}
Combining this estimate with (\ref{newstab}) and Parseval's identity gives 
\begin{equation} \label{halfway}
\|P_{D,f} -P_{D,f_0}\|_{H^2 \to H^2}\gtrsim \|G\|_{L^2} = \|\nabla \cdot [(f-f_0)\nabla E_{1,f_0,\iota}]\|_{L^2}.
\end{equation}
The theorem then follows from Lemma \ref{transport} with $u_0 = E_{1,f_0,\iota}$ which satisfies (\ref{buddy}) by hypothesis (\ref{sunnyside}) and is supremum-norm bounded by (\ref{weylsup}).

\subsubsection{Stability of a transport operator}

We now give a stability lemma for the operator $$T(h)= \nabla \cdot (h \nabla u_0), h \in C^1,$$ for appropriate choices of $u_0$. It  features regularly in stability estimates for elliptic PDEs, see Chapter 2 in \cite{N23} for references. 

\begin{condition}\label{falco} Let $u_0 \in H^2(\mathcal O)$ be a function such that $\sup_{x \in \mathcal O_0}|u_0(x)| \le u<\infty$ and
	\begin{equation} \label{buddy}
			\frac{1}{2}\Delta u_0(x) + \mu |\nabla u_0(x)|^2 \ge c_0>0,~a.e.~ x \in \mathcal O_0,
	\end{equation}
for some compact subset $\mathcal O_0$ of $\mathcal O$. 
\end{condition}

\begin{lemma}\label{transport}
For $u_0$ as in Condition \ref{falco} and any $h \in C^1$ that vanishes on $\mathcal O \setminus \mathcal O_0$, the operator $T(h)$ satisfies $\|\nabla\cdot (h \nabla u_0)\|_{L^2(\mathcal O)} \ge \underline c \|h\|_{L^2(\mathcal O)}$ for a constant $\underline c = \underline c(u,c_0, \mu)>0$.
\end{lemma}
\begin{proof}The divergence theorem applied to any $v \in H^2(\mathcal O)$ vanishing at $\partial \mathcal O$  gives $\langle \Delta u_0, v^2 \rangle_{L^2} + \frac{1}{2} \langle \nabla u_0, \nabla (v^2) \rangle_{L^2} = \frac{1}{2}\langle \Delta u_0, v^2 \rangle_{L^2}.$ For $v=e^{-\mu u_0}h$ with $\mu>0$ from (\ref{buddy})
		\begin{equation*}
		\frac{1}{2} \int_{\mathcal O}  \nabla (v^2) \cdot \nabla u_0 =- \int_{\mathcal O} \mu |\nabla u_0|^2 v^2 + \int_\mathcal O v e^{-\mu u_0}\nabla h \cdot \nabla u_0,
		\end{equation*}
		so that by the Cauchy-Schwarz inequality
		\begin{align} \label{keylb}
		 \left|\int_{\mathcal O}\Big(\frac{1}{2}\Delta u_0+\mu |\nabla u_0|^2\Big)v^2 \right| &= \left|\langle (\Delta u_0 + \mu |\nabla u_0|^2), v^2 \rangle_{L^2} + \frac{1}{2} \langle \nabla u_0, \nabla (v^2)\rangle_{L^2}\right| \notag \\
		&= \left|\langle h \Delta u_0 + \nabla h \cdot \nabla u_0, h e^{-2\mu u_0} \rangle_{L^2} \right| \le \bar \mu \|\nabla \cdot (h \nabla u_0)\|_{L^2}  \|h\|_{L^2} 
		\end{align}
		for $\bar \mu=\exp(2\mu \|u_0\|_{\infty})$.  Now by (\ref{buddy}) and since $h=0=v$ on $\mathcal O \setminus \mathcal O_0$ by hypothesis we have
		$$ \left|\int_{\mathcal O}\Big(\frac{1}{2}\Delta u_0+\mu |\nabla u_0|^2\Big)v^2 \right| =  \left|\int_{\mathcal O_0}\Big(\frac{1}{2}\Delta u_0+\mu |\nabla u_0|^2\Big)v^2 \right|  \ge c_0 \int_{\mathcal O_0} v^2$$ and fusing also (\ref{keylb}) we deduce $\|\nabla \cdot (h \nabla u_0)\|_{L^2}\|h\|_{L^2}\geq c'\|v\|^2_{L^2(\mathcal O_0)}\ge \underline c \|h\|^2_{L^2(\mathcal O)}$.
\end{proof}

\begin{lemma} \label{zauberstab}
Let $\mathcal O_0$ be any compact subset of a bounded smooth domain $\mathcal O$ and suppose that $f_1, f_2$ are two $C^2$-diffusivities $f_i \ge f_{min} >0, i=1,2,$ such that $f_1=f_2$ on $\mathcal O \setminus \mathcal O_0$. Suppose for some $\phi \in C^\infty(\mathcal O) \cap L^2_0(\mathcal O)$ s.t.~$\phi \ge 1$ on $\mathcal O_0$, functions $u_{f_i}, i=1,2,$ solve
\begin{align}\label{aeneumann}
\nabla \cdot (f_i \nabla  u_{f_i}) &= \phi ~~\text{ on } \mathcal O \\
\frac{\partial u_{f_i}}{\partial \nu} & = 0 ~~\text{ on } \partial \mathcal O. \notag
\end{align}
Then we have for some constant $C=C(\|\phi\|_\infty, \|f_1\|_{C^1})>0$ that
\begin{equation}\label{zaubineq}
\|f_1 - f_2\|_{L^2(\mathcal O)} \le C \|f_2\|_{C^1} \|u_{f_1} - u_{f_2}\|_{H^2}.
\end{equation}
\end{lemma}
\begin{proof}
Let us write $h=f_1-f_2$. By (\ref{aeneumann}), we have on $\mathcal O$
		\begin{equation}\label{div-diff}
		\begin{split}
		\nabla \cdot (h \nabla u_{f_1}) &= \nabla\cdot  (f_1 \nabla u_{f_1}) - \nabla\cdot (f_2 \nabla u_{f_2}) - \nabla\cdot  (f_2 \nabla (u_{f_1}-u_{f_2})) \\
		&= \nabla \cdot (f_2 \nabla (u_{f_2}-u_{f_1})).
		\end{split}
		\end{equation}
We can upper bound the $\|\cdot\|_{L^2}$-norm of r.h.s. by 
		\begin{align}
		\|\nabla \cdot (f_2 \nabla (u_{f_2}-u_{f_1}))\|_{L^2}&\leq \|\nabla f_2\|_{\infty}\|u_{f_2}-u_{f_1}\|_{H^1} +\|f_2\|_{\infty}\|u_{f_2}-u_{f_1}\|_{H^2}\notag \\
		&\leq 2\|f_2\|_{C^1}\|u_{f_2}-u_{f_1}\|_{H^2}. \label{diff-ub}
		\end{align}
To lower bound the left hand side of (\ref{diff-ub}) we apply Lemma \ref{transport} with $u_0 = u_{f_1}$ to $\|\nabla \cdot (h \nabla u_{f_1})\|_{L^2}$. The hypothesis on $\phi$ implies $1 \le  f_1  \Delta u_{f_1} + \nabla f_1 \cdot \nabla u_{f_1}$ on $\mathcal O_0$, so that either $\Delta u_{f_1}(x) \ge 1/2\|f_1\|_\infty$ or $|\nabla u_{f_1}(x)|^2 \ge (1/2\|f_1\|_{C^1})^2$ for $x \in \mathcal O_0$. Since $\|u_{f_1}\|_\infty + \|\Delta u_{f_1}\|_\infty\lesssim \|u_{f_1}\|_{C^2} \lesssim c(\|f_1\|_{C^1})$ by a $C^\alpha$-regularity estimate (e.g., Thm 4.3.4 in \cite{T83}) for solutions of (\ref{aeneumann}) with $f_1 \in C^1$ this implies (\ref{buddy}) and by Lemma \ref{transport}  the result.
\end{proof}

\subsection{Minimax estimation of the transition operator $P_{D,f}$}

\subsubsection{Operator norm convergence}

In this subsection we construct explicit estimator $\hat P_D$ for the transition operator $P_{D,f}$ and prove Theorem \ref{opmx}. While it is possible to take $\hat P_D$ self-adjoint, this will not be required here. 

For $J \in \mathbb N$  take $E_J \equiv \{e_{j,1}: 0 \le j \le J-1\}$ the eigenfunctions of the Neumann Laplacian $\mathcal L_1$ on $\mathcal O$ (including $e_0=1$) and regard $E_J \simeq \mathbb R^J$ as a normed space equipped with the Euclidean norm via Parseval's identity for $L^2(\mathcal O)$. Given the observations $X_0, X_D, \dots, X_{ND}$ define a $J\times J$ matrix by 
\begin{equation}
\hat {\mathbf P}_{j,j'} = \frac{1}{N} \sum_{i=1}^N e_{j,1}(X_{(i-1)D}) e_{j',1}(X_{iD}),~~~~~0 \le j,j' \le J-1.
\end{equation}
Via the injection of $E_J \simeq \mathbb R^J$ into $L^2(\mathcal O)$ we can regard $\hat {\mathbf P}_{J}$ as a bounded linear operator $\hat P_{J}$ on $L^2$ described by the action
\begin{align}\label{hatp}
\langle \hat P e_{j,1}, e_{j',1} \rangle_{L^2} &\equiv \hat {\mathbf P}_{j,j'},~~~~~0 \le j,j' \le J-1 \notag \\
 &=0~~~~~~~~~~~~~~~~~~~~~~~~~\text{ otherwise }.
\end{align}
Similarly the transition operator $P_{D,f}$ induces a matrix ${\mathbf P}_{D, f, J}$ via
\begin{align*}
{\mathbf P}_{D,f, j,j'} =  \langle P_{D,f} e_{j,1}, e_{j',1} \rangle_{L^2} &= \mathbb E_f e_{j,1}(X_0) e_{j',1}(X_D), ~~~0 \le j,j' \le J-1, \\
&= 0~~~~~~~~~~~~~~~~~~~~~~~~~\text{ otherwise,}
\end{align*}
which equals the expectation $\mathbb E_f\hat {\mathbf P}_{D, J}  = {\mathbf P}_{D,f, J}$ under the law $\mathbb P_f$ of $(X_t: t \ge 0)$ started in stationarity $X_0 \sim Unif(\mathcal O)$. The latter matrix corresponds to the operator on $L^2$ arising from the composition $\pi_{E_J} P_{D,f}$ where $\pi_{E_J}$ describes the projection onto $E_J$ -- note that $E_J$ are not the eigen-spaces of $P_{D,f}$ unless $f=1$. To obtain an estimate for the approximation error from $E_J$, note first that by Proposition \ref{sobald} and (\ref{sobnorm}), (\ref{langweylig}), for any $\phi \in L^2_0$ s.t.~$\|\phi\|_{L^2} \le 1$,
\begin{equation}
\|P_{D,f}(\phi)\|^2_{\bar H_{1}^{s+1}} \lesssim \|P_{D,f}(\phi)\|^2_{\bar H_f^{s+1}}  = \sum_{j \ge 1} e^{-2D \lambda_{j,f}} \lambda^{s+1}_{j,f} \langle \phi, e_{j,f} \rangle_{L^2}^2 \le B'
\end{equation}
for some $B'=B'(U)<\infty$ since $\|f\|_{H^s} \le U$ by hypothesis. Therefore, using again (\ref{sobnorm}), (\ref{langweylig}) and Parseval's identity
\begin{align}\label{bias}
\|\pi_{E_J} P_{D,f} - P_{D,f}\|_{L^2 \to L^2} & = \sup_{\phi \in L^2_0, \|\phi\|_{L^2} = 1} \|\pi_{E_J} P_{D,f}(\phi) - P_{D,f}(\phi)\|_{L^2}  \\
\leq \sup_{\|\psi\|_{\bar H_1^{s+1}} \le B'} \|\pi_{E_J}\psi - \psi\|_{L^2}& \le  \sup_{\|\psi\|_{\bar H_1^{s+1}} \le B'} \sqrt{\sum_{j>J} \lambda_{j,1}^{-s-1} \lambda_{j,1}^{s+1} \langle \psi, e_{j,1} \rangle_{L^2}^2}  \lesssim J^{-(s+1)/d}. \notag
\end{align}
To bound the operator norms on approximation spaces $E_J \simeq \R^J$ we use a standard covering argument in finite dimensional spaces (e.g., the proof of Lemma 1.1 in \cite{CP11}) to the effect that
$$\|\hat P_J - \pi_{E_J}P_{D,f}\|_{L^2 \to L^2}=\|\hat P_J - \pi_{E_J}P_{D,f}\|_{E_J \to E_J} \le
2\max_{\mathbf{u,v} \in D_J(1/4)}\big|\mathbf {u}^T(\hat {\mathbf P}_J - \mathbf P_{D,f})\mathbf{v}\big|$$ where $D_J(1/4)$ is a discrete $1/4$-net of unit vectors (i.e., $\|\mathbf{u}\|_{\R^J}=1$) covering the unit sphere of $\R^J$ of cardinality at most $card(D_J(1/4))\leq A^J$ for some $A>0$, see, e.g., \cite{GN16}, p.373. By a union bound and for $g(x,y)=u(x)v(y)$ with $u=\sum_j \mathbf {u_j} e_j$, $v=\sum_j \mathbf {v_j} e_j$, we obtain 
\begin{align*}
& \mathbb P_f \left(\|\hat P_J - \pi_{E_J}P_{D,f}\|_{L^2 \to L^2} > c\sqrt{\frac{J}{N}} \right) \\
 & \leq A^J \max_{\mathbf{u,v} \in D_J(1/4)}  \mathbb P_f \left(\big|\mathbf {u}^T(\hat {\mathbf P}_J - \mathbf P_{D,f})\mathbf{v}\big| > c\sqrt{\frac{J}{4N}} \right) \\
&= A^J \max_{g} \mathbb P_f \left(\big|\sum_{i=1}^N g(X_{(i-1)D}, X_{iD}) - \mathbb E_f g(X_{(i-1)D}, X_{iD})  \big| > c\sqrt{JN/4} \right)
\end{align*}
We can apply the concentration inequality Proposition \ref{bernie} below with $h=g-\int_\mathcal P g p_{D,f}$ an element of the Hilbert space $L^2_0(P_{D,f})$ from (\ref{hnull}) below. We have, using also Proposition \ref{transreg}, $$\|h\|_{L^2(P_{D,f})} \lesssim \|h\|_{L^2(\mathcal O \times \mathcal O, dx \otimes dx)} \le C$$  as well as $\|h\|_\infty \le H \lesssim J^{2\tau +1}$ in view of the estimate
$$\|u\|_\infty \le \|\mathbf u\|_{E_J} \sqrt{\sum_{j \le J} \|e_j\|_\infty^2} \lesssim J^{\tau +1/2},~~ \tau>1/2,$$ where we have used (\ref{weylsup}). In this way we obtain overall:  

\begin{proposition}\label{opconc}
Let $D>0$ and suppose $X_0, X_D, \dots, X_{ND}$ arise from the diffusion (\ref{diffus}) started at $X_0 \sim Unif(\mathcal O)$ on a bounded smooth convex domain $\mathcal O$ with $f: \mathcal O \to [f_{min}, \infty), f_{min}>0,$ s.t.~$\|f\|_{H^s} + \|f\|_{C^2} \le U, s>d$. Let $J>0$ be s.t. $J^{\bar \tau} \lesssim \sqrt N$ for some $\bar \tau >5/2$. Then for all $c>0$ we can choose $C=C(U,D)>0$ such that \begin{equation}\label{oprate}
\mathbb P_f\left(\|\hat P_J - P_{D,f}\|_{L^2 \to L^2} \ge C \Big(\sqrt{\frac{J}{N}} + J^{-(s+1)/d}\Big) \right) \le e^{-cJ}.
\end{equation}
\end{proposition}

In particular for $s>2d-1$ we can choose $J \simeq N^{d/(2s+2+d)}$ to prove Theorem \ref{opmx}. A bound on the $H^2 \to H^2$-operator norms follows as well: Since the imbedding $H^2 \subset L^2$ is continuous and since $\|v\|_{H^2} \simeq \|v\|_{\bar H^2_1} \lesssim J^{2/d} \|v\|_{L^2}$ whenever $v \in E_J$, we have 
$$\|\hat P_D - \pi_{E_J}P_{D,f}\|_{H^2 \to H^2} \lesssim J^{2/d} \|\hat P_D - \pi_{E_J}P_{D,f}\|_{L^2 \to L^2}$$
and as in (\ref{bias}) and by Proposition \ref{sobald} the approximation errors scale like
$$\|\pi_{E_J} P_{D,f} - P_{D,f}\|_{H^2 \to H^2} \lesssim \sup_{\|\psi\|_{\bar H_1^{s+1}} \le B'} \|\pi_{E_J}\psi - \psi\|_{H^2} \lesssim J^{(s-1)/d}.$$ 
\begin{corollary}\label{opconc2}
In the setting of Proposition \ref{opconc} we also have 
\begin{equation}\label{oprate2}
\mathbb P_f\left(\|\hat P_D - P_{D,f}\|_{H^2 \to H^2} \ge C \Big(J^{2/d}\sqrt{\frac{J}{N}} + J^{(s-1)/d} \Big)\right) \le e^{-cJ},
\end{equation}
\end{corollary}

\subsubsection{A concentration inequality for ergodic averages}

Consider the discrete Markov chain $X_D, \dots, X_{ND}$ arising from sampling the diffusion (\ref{diffus}) started in stationarity $X_0 \sim Unif(\mathcal O)$. The transition operator of this chain is $P_{D,f}$ from (\ref{trans}), with spectrum $1 > e^{-D\lambda_{1,f}} \ge e^{-D\lambda_{2,f}} \ge \dots$ and the first spectral gap is bounded as
\begin{equation} \label{gap}
1 -  e^{-D\lambda_{1,f}} \ge r_D
\end{equation}
in view of (\ref{topevar}) for some $r_D=r(D,f_{min}, p_\mathcal O, U)>0$. We initially establish concentration bounds for additive functionals
$$\sum_{i=1, i \text{ odd }}^N h(X_{(i-1)D}, X_{iD}),~~\text{ and } \sum_{i=1, i \text{ even }}^N h(X_{(i-1)D}, X_{iD}),~~~h: \mathcal O \times \mathcal O \to \mathbb R,$$
of bivariate Markov chains in $\mathcal O \times \mathcal O$ arising from $$(X_0, X_D), (X_{2D}, X_{3D}), (X_{4D}, X_{5D}), \dots,~~\text{ and }(X_D, X_{2D}), (X_{3D}, X_{4D}), (X_{5D}, X_{6D}), \dots,$$ respectively. By a union bound this will give concentration inequalities for ergodic averages $\sum_{i=1}^N h(X_{(i-1)D}, X_{iD})$ along all indices $i$, see (\ref{conc}) below.

The transition operators $P_{D,f}'$ of the new bivariate Markov chains have invariant measure $p_{D,f}(x,y)$ on $\mathcal O \times \mathcal O$. If we define 
\begin{equation} \label{hnull}
L^2_0(P_{D,f}) := \left \{h:  \int_\mathcal O \int_\mathcal O h(x,y) p_{D,f}(x,y) dx dy = 0 \right\}
\end{equation}
 then one shows
\begin{equation}
\sup_{h: \int h p_{D,f} =0}\frac{\|P_{D,f}'[h]\|_{L^2(P_{D,f})}}{\|h\|_{L^2(P_{D,f})}} \le \sup_{h: \int h =0} \frac{\|P_{D,f}[h]\|_{L^2}}{\|h\|_{L^2}} \le e^{-D\lambda_{1,f}}
\end{equation}
by a basic application of Jensen's inequality (cf.~Lemma 24 in \cite{NS17}), and by (\ref{gap}). By the variational characterisation of eigenvalues and (\ref{gap}) this implies that the first spectral gap $\rho_D$ of $P_{D,f}'$ is also bounded as
\begin{equation}\label{gapp}
\rho_D = 1 -   e^{-D\lambda_{1,f}} \ge r_D.
\end{equation}
We deduce from Theorem 3.1 in \cite{P15} that for any $h \in L^2_0(P_{D,f})$ we have the variance bound
\begin{equation} \label{varbd}
Var_f\Big(\frac{1}{N}\sum_{i=1, i \text{ odd }}^N h(X_{(i-1)D}, X_{iD})\Big) \le \frac{2}{N\rho_D} \|h\|_{L^2(P_{D,f})}^2 \le \frac{1}{Nr_D} \|h\|^2_{L^2(P_{D,f})}
\end{equation}
where we have also used (\ref{gapp}). Similarly, requiring in addition $\|h\|_\infty \le H$, Theorem 3.3 and eq.~(3.21) in \cite{P15} imply the concentration inequality. 
\begin{equation} \label{concodd}
\mathbb P_f \Big(\sum_{i=1, i \text{ odd }}^N h(X_{(i-1)D}, X_{iD}) \ge x\Big) \le 2 \exp\left\{-\frac{x^2 r_D}{4N\|h\|^2_{L^2(P_{D,f})} + 10x H} \right\},~~~x >0.
\end{equation}
The same inequality applies to the even indices $i$, so that by a union bound we obtain:

\begin{proposition}\label{bernie}
Let $h \in L^2_0(P_{D,f})$ be s.t.~$\|h\|_\infty \le H$, and let $X_0, X_D, \dots, X_{ND}$ be sampled discretely at observation distance $D>0$ from the diffusion $(X_t: t \ge 0)$ from (\ref{diffus}) with $f_{min} \le f \le U<\infty$. Then for some constant $c=c(r,D)$ and all $x>0$ we have
\begin{equation} \label{conc}
\mathbb P_f\Big(\sum_{i=1}^N h(X_{(i-1)D}, X_{iD}) \ge x\Big) \le 4 \exp\left\{-c\frac{x^2}{N\|h\|^2_{L^2(P_{D,f})} + x H} \right\}.
\end{equation}
\end{proposition}

\subsubsection{Proof of the minimax lower bound Theorem \ref{opmxlb}}

Given the analytical estimates obtained so far, the proof follows ideas of the lower bound Theorem 10 of \cite{NvdGW20} and we sketch here only the necessary modifications. Let us take the same set of functions $(f_m:m =1, \dots, M), f_0=1,$ from (4.17) in \cite{NvdGW20} and consider only $j$ large enough in that construction such that all the wavelets featuring there are contained inside of the compact subset $\mathcal O_0$ of the `smoothed' $d$-dimensional hypercube $\mathcal O \equiv \mathcal O_{m,w}$ from (\ref{om}) for $m,w$ from Theorem \ref{cylinderth}B). In particular we can choose $j$ so large that  $\|f_m - 1\|_\infty < \kappa$ for the $\kappa$ from Theorem \ref{cylinderth}B). We apply Theorem 6.3.2 in \cite{GN16} (taking also note of (6.99) there to obtain an `in probability version' of the lower bound) as in Step VII of the proof of Theorem 10 in \cite{NvdGW20}, noting that in our setting we can control the KL-divergences $KL(f_m, f_0) \lesssim \|f_m -f_0\|_{H^{-1}}$ via Theorem \ref{lfsball} and the imbedding $(H^1(\mathcal O))^* \subset H^{-1}(\mathcal O)$. The result will thus follow if we can show that the transition operators induced by the $P_{D,f_m}$'s are appropriately separated for the $H^2$-operator norms. Using the inequality (\ref{halfway}) we have
$$\|P_{D,f_m} - P_{D,f_{m'}}\|_{H^2 \to H^2} \gtrsim \|\nabla \cdot [(f_m-f_{m'}) \nabla e_{1,f_{m'}}]\|_{L^2},~~1 \le m, m' \le M,$$ 
where we note that on our `smoothed' cylinder, the eigenfunctions $e_{1,f_{m'}}$ are all simple thanks to Theorem \ref{cylinderth}B). To proceed we need to lower bound the $L^2$-norms of the r.h.s.~of (4.19) in \cite{NvdGW20}, with $u_{f_{m'}}$ there replaced by our $e_{1,f_{m'}}$. As will be shown in the proof of Proposition \ref{cyllapsimp}, the first eigenfunction $e_{1,1}$ of $\Delta$ on $[0,1]^{d-1} \times (0,w)$ has all partial derivatives equal to zero except  with respect to one, say the first, variable, and that partial derivative cannot vanish on $\mathcal O_0$. In view of (\ref{c2ef}), (\ref{intpolef}) this implies that the corresponding eigenfunction $e_{1, f_m'}$ on $\mathcal O$ has a partial derivative for the first variable that is strictly positive while the other partial derivatives are bounded (in fact can be made arbitrarily close to zero). One can then easily adapt the steps V and VI in the proof of Theorem 10 in \cite{NvdGW20} (with $\varepsilon^{-2}$ there equal to our $N$) to establish, for all $N$ large enough, the required bound
$\|\nabla \cdot [(f_m-f_{m'}) \nabla e_{1,f_{m'}}]\|_{L^2} \gtrsim N^{-(s-1)/(2s+2+d)}.$

\subsection{Bayesian contraction results}

\subsubsection{Results for general priors}

In this subsection we follow general ideas from Bayesian nonparametrics \cite{GV17} and specifically in our diffusion context adapt the results from \cite{NS17} to our multi-dimensional setting to obtain a contraction theorem for posteriors arising from general possibly $N$-dependent priors $\Pi$. Recall the information distance $KL$ from (\ref{KLpobs}) on parameter spaces $\mathcal F \subset C^2(\mathcal O) \cap \{f \ge f_{min}\}, f_{min}>0$.

\begin{lemma}\label{elbo}
For $\delta>0$ define $$B_\delta = \Big\{f \in \mathcal F: KL(f, f_0) \le \delta^2,~~Var_{f_0}\Big(\log \frac{p_{D,f}(X_0, X_D)}{p_{D,f}(X_0,X_D)} \Big) \le 2 \delta^2 \Big\}.$$ Then for any probability measure $\nu$ on $B_\delta$, any $c>0$ and $\rho_D \in [0,r_D]$  from (\ref{gapp}),
\begin{equation*}
\mathbb P_{f_0}\left(\int_{B_\delta}\prod_{i=1}^N \frac{p_{D,f}(X_{(i-1)D}, X_{iD})}{p_{D, f_0}(X_{(i-1)D}, X_{iD})} d\nu(f) \le \exp\{-(1+c)N\delta^2\}\right) \le \frac{6 (1+\rho_D)}{c^1(1-\rho_D)N\delta^2}.
\end{equation*}
\end{lemma}
\begin{proof}
The proof is the same as the one Lemma 25 in \cite{NS17}, ignoring the term involving invariant measures $\mu_{\sigma, b}$ there as in our case $\mu_f = \mu_{f_0} =const$ for all $f$. The key variance estimate in that lemma can then be replaced by our (\ref{varbd}) with $h=\log \frac{p_{D,f}(X_0, X_D)}{p_{D,f_0}(X_0,X_D)}$. 
\end{proof}

\begin{theorem}\label{notagain}
Let $\Pi=\Pi_N$ be a sequence of priors on $\mathcal F$ and suppose for $f_0 \in \mathcal F$, some sequence $\delta_N \to 0$ such that $\sqrt N \delta_N \to \infty$ and constant $A>0$ we have
\begin{equation}\label{smball}
\Pi_N(B_{\delta_N}) \ge e^{-A N \delta_N^2}.
\end{equation}
Suppose further for a sequence of subsets $\mathcal F_N \subset \mathcal F$ and constant $B>A+2$ we have 
\begin{equation}\label{except}
\Pi_N(\mathcal F \setminus \mathcal F_N) \le e^{-B N \delta_N^2}
\end{equation}
and that there exists tests $\Psi_N = \Psi(X_0, \dots, X_{ND})$ and a sequence $\bar \delta_N \to 0$ such that 
\begin{equation}\label{tests}
\mathbb E_{f_0}\Psi_N \to_{N \to \infty} 0,~~~\sup_{f \in \mathcal F_N, d(f,f_0)> \bar \delta_N} \mathbb E_{f}[1-\Psi_N] \le e^{-B N \delta_N^2},
\end{equation}
where $d$ is some distance function on $\mathcal F$.
Then we have for $0<b<B-A-2$ that
\begin{equation}\label{contract}
\Pi\big(\mathcal F_N \cap \{f: d(f, f_0) \le \bar \delta_N\} |X_0, \dots, X_{ND}\big) = 1 - O_{\mathbb P_{f_0}}(e^{-bN \delta_N^2}).
\end{equation}
\end{theorem}
\begin{proof}
The proof is the same as the one of Theorem 13 in \cite{NS17}. We can track the constants in this proof (similar as in Theorem 1.3.2 in \cite{N23}) to further include the set $\mathcal F_N$ in, and to obtain the explicit convergence rate bound on the r.h.s.~of, (\ref{contract}).
\end{proof}

\subsubsection{Proof of Theorems \ref{maint} and \ref{main}}

With these preparations we can now prove Theorem \ref{maint} and a version of it with distance functions $d(f,f_0)=\|P_{D,f}-P_{D,f_0}\|_{L^2\to L^2}$ replaced by $d(f,f_0)=\|P_{D,f}-P_{D, f_0}\|_{H^2\to H^2}$, relevant to prove Theorem \ref{main}. We will choose $$\delta_N = MN^{-(s+1)/(2s+2+d)}$$ throughout, for $M$ a large enough constant. We consider the prior $\Pi_N$ from (\ref{gp}) and use standard theory for Gaussian processes (e.g., Ch.2 in \cite{GN16}). In particular, recalling the cut-off function $\zeta$, we note that the reproducing kernel Hilbert space (RKHS) $\mathbb H_N$ of the Gaussian process $\theta$ generating $\Pi_N$ is given by $\mathbb H_N = \{\zeta h: h \in E_K\} \subset C^\infty_c$, with RKHS norm
\begin{equation}\label{RKHS}
\|g\|_{\mathbb H_N} \simeq \sqrt N\delta_N\big( |\langle \zeta^{-1} g, 1\rangle_{L^2}| + \|\zeta^{-1}(g-\langle g, 1\rangle_{L^2})\|_{\bar H^s_1}\big),~~ g \in \mathbb H_N.
\end{equation}

\textbf{i) Verification of (\ref{smball}).} Proposition \ref{transreg} with $k>d/2$ and Proposition \ref{translow} imply the two sided estimate $0<c_{lb} \le p_{D,f}(x,y) \le c_{ub}<\infty$ with constants that are uniform in $\|f\|_{H^s} \le U$. This applies as well to $f_0 \in H^s$ and so, by standard inequalities (e.g., Appendix B in \cite{GV17}),
$$E_{f_0}\Big|\log \frac{p_{D,f}(X_0, X_D)}{p_{D,f_0}(X_0, X_D)}\Big|^2 \lesssim \|p_{D,f}-p_{D,f_0}\|^2_{L^2(\mathcal O \times \mathcal O)} = \|P_{D,f}-P_{D,f_0}\|^2_{HS}$$ for such $f$, with constants depending on $U,s,d, \mathcal O$. 

Let us define $\theta_0 = \log (4f_0 -1)$ which is zero outside of $\mathcal O_{00}$ and lies in $H^s_c$ by the hypotheses on $f_0$. This implies that $\theta_0 - \langle \theta_0,1 \rangle_{L^2} \in H^s_c/\R \cap L_0^2 \subset \bar H^s_1$ by Proposition \ref{sobald}. If $\theta_{0,K}$ is the $L^2$-projection of $\theta_0$ onto $E_K$, then $\zeta \theta_{0,K} \in \mathbb H_N$ and 
\begin{equation}\label{boring}
\|\zeta \theta_{0,K}\|_{H^s} \lesssim |\langle \theta_0,1 \rangle_{L^2}| + \|\theta_{0,K}- \langle \theta_0,1 \rangle_{L^2}\|_{\bar H^s_1} \lesssim \|\theta_0\|_{H^s} \lesssim U.
\end{equation}
Since $H^1_c/\R \cap L^2_0 \subset \bar H^1_{1}$ (Proposition \ref{sobald}) implies that $\bar H^{-1}_1$ embeds continuously into $(H^1_c/\R \cap L^2_0)^*$, we can use (\ref{langweylig}) and choose $M$ large enough s.t.
\begin{align*}
\|\theta_0 - \zeta \theta_{0,K}\|_{(H_c^1)^*} &= \|\zeta(\theta_0-\theta_{0,K})\|_{(H^1_c)^*} \lesssim \|\zeta\|_{C^1} \|\theta_0 - \theta_{0,K}\|_{(H^1_c /\R \cap L^2_0)^*} \\
&\lesssim \|\theta_0 - \theta_{0,K}\|_{\bar H^{-1}_1} =\Big(\sum_{j>K} \lambda_{j,1}^{-1}\langle\theta_{0}, e_{j,1} \rangle_{L^2}^2\Big)^{1/2} \lesssim K^{-(s+1)/d}U\le c\delta_N/2
\end{align*}
 for any given $c, U>0$. Now using Theorem \ref{lfsball}, (\ref{boring}) and for $C_i>0$, with $M,B$ large enough,
\begin{align*}
\Pi_N(B_{\delta_N}) &\ge \Pi_N(\{\|f_\theta-f_{\theta_0}\|_{(H_c^1)^*} \le C_1 \delta_N\} \cap \{\theta: \|\theta\|_{H^s} \le 2B\}) \\
& \ge \Pi_N\big(\theta: \|\theta-\theta_0\|_{(H^1_c)^*} \le C_2 \delta_N, \|\theta-\zeta\theta_{0,K}\|_{H^s} \le B\}\big) \\
& \ge  \Pi_N\big(\theta: \|\theta-\zeta \theta_{0,K}\|_{(H^1_c)^*} \le C_3 \delta_N, \|\theta-\zeta\theta_{0,K}\|_{H^s} \le B\}\big)
\end{align*}
where we have used that the map $\theta \mapsto e^\theta$ is Lipschitz on bounded sets of $H^s$ for the $(H^1_c)^*$-norm (cf.~the argument on p.27 in \cite{N23}). We apply Corollary 2.6.18 in \cite{GN16} with `shift' vector $\zeta \theta_{0,K} \in \mathbb H_N$ and the Gaussian correlation inequality (in the form of Theorem B.1.2 in \cite{N23}) to further lower bound the r.h.s.~in the last display by
\begin{align*}
& \ge e^{-\|\zeta \theta_{0,K}\|_{\mathbb H_N}^2/2} \Pi_N\big(\theta: \|\theta\|_{(H^1_c)^*} \le C_3 \delta_N, \|\theta\|_{H^s} \le B\}\big) \\
&\ge e^{-\tilde c N \delta_N^2} \Pi_N(\|\theta\|_{(H^1_c)^*} \le C_3 \delta_N) \Pi_N(\|\theta\|_{H^s} \le B),
\end{align*}
using also (\ref{RKHS}), (\ref{boring}) and for some $\tilde c=c(U)>0$.  Next, since the RKHS of the base prior $\theta'= \sqrt N \delta_N \theta$ embeds continuously into $H^s_c \subset H^s_0$ (cf.~(\ref{RKHS})), we obtain
\begin{equation}\label{small}
\Pi_N(\|\theta\|_{(H^1_c)^*} \le C_3 \delta_N)=\Pi_N(\|\theta'\|_{(H^1_c)^*} \le C_3 \sqrt N \delta^2_N) \ge e^{-aN\delta_N^2}
\end{equation}
as in eq.~(2.24) in \cite{N23} with $\kappa=1$ there. In concluding this step we now also construct the regularisation sets $\mathcal F_N$ for (\ref{except}). If we define
\begin{align*}
\Theta_N = \big\{\theta =\zeta \vartheta,& \vartheta \in E_K, \vartheta =\vartheta_1 + \vartheta_2, |\langle \vartheta_1, 1 \rangle_{L^2}|+\|\vartheta_1-\langle \vartheta_1,  1\rangle_{L^2}\|_{\bar H^{-1}_1} \le m \delta_N, \\ & |\langle \vartheta_2, 1 \rangle_{L^2}|+\|\vartheta_2-\langle \vartheta_2, 1\rangle_{L^2}\|_{\bar H^s_1} \le m\big\}
\end{align*} 
then for every $B$ we can choose $m$ large enough so that  $\Pi_N(\Theta_N) \ge 1-e^{-B N \delta_N^2}$, by an application of the Gaussian isoperimetric theorem \cite{GN16} as in step iii) in the proof of Theorem 2.2.2 in \cite{N23} with $\kappa=1$. Now we have
$\|\theta\|_{H^s} = \|\zeta \vartheta\|_{H^s} \lesssim \|\vartheta - \langle \vartheta, 1 \rangle_{L^2}\|_{H^s} + | \langle \vartheta_1, 1 \rangle_{L^2}| +  | \langle \vartheta_2, 1 \rangle_{L^2}|$ and the last two terms are bounded by $2m$ for $\theta \in \Theta_N$. For the first we can use Proposition \ref{sobald} and the triangle inequality to obtain on $\Theta_N$
\begin{align*}
 \|\vartheta - \langle \vartheta, 1 \rangle_{L^2}\|_{H^s} &\lesssim  \|\vartheta - \langle \vartheta, 1 \rangle_{L^2}\|_{\bar H_1^s} \le \|\vartheta_1 - \langle \vartheta_1, 1 \rangle_{L^2}\|_{\bar H_1^s}+ \|\vartheta_2 - \langle \vartheta_2, 1 \rangle_{L^2}\|_{\bar H_1^s} \le c+m.
 \end{align*}
where we have used (\ref{langweylig}) in the estimate
\begin{align*}
 \|\vartheta_1 - \langle \vartheta_1, 1\rangle_{L^2}\|^2_{\bar H^s_1} &= \sum_{1 \le j \le K}  \frac{\lambda_{j,1}^{s+1}}{\lambda_{j,1}} \langle \vartheta_1, e_{j,1} \rangle_{L^2}^2  \lesssim K^{\frac{2(s+1)}{d}} \|\vartheta_1-\langle \vartheta_1, 1\rangle_{L^2}\|^2_{\bar H^{-1}_1} \lesssim N^{\frac{2s+2}{2s+2+d}} \delta_N^2,
\end{align*}
and the last term is bounded by a fixed constant $c^2$.  In conclusion this proves $\Pi_N(\|\theta\|_{H^s} \le B) \ge 1/2$ for all $B',N$ large enough so that (\ref{smball}) follows for our choice of $\delta_N, A>a+\tilde c,$ and all $M$ large enough. Since $\theta \mapsto f_\theta$ is Lipschitz on bounded subsets of $H^s$, we have in fact proved the stronger result -- to be used in the next step -- that for some $U>0$ we have
\begin{equation}\label{regiso}
\mathcal F_N := \{f_\theta: \theta \in \Theta_N\} \subset \{f: \|f\|_{H^s} \le U\}, ~\Pi_N(\mathcal F \setminus \mathcal F_N) \le e^{-B N \delta_N^2}.
\end{equation}

\textbf{ii) Construction of tests.}  We cannot rely on Hellinger testing theory as in \cite{GV17, MNP21, N23} because our data does not arise from an i.i.d.~model. Instead (following ideas from  \cite{GN11, NS17}) we use concentration inequalities, specifically Proposition \ref{opconc}, to construct these tests. For the hypothesis $H_0: f=f_0$ consider the plug in test $\Psi_N = 1\{\|\hat P_D - P_{D,f_0}\|_{L^2 \to L^2} \ge  M \delta_N \},$ where $\hat P_D$ is from (\ref{hatp}) with choice $J=BN\delta_N^2$. We verify (\ref{tests}) with $\mathcal F_N$ from (\ref{regiso}). By Proposition \ref{opconc}, the type-one error is then controlled, for $M$ large enough, as
$$\mathbb E_{f_0}\Psi_N = \mathbb P_{f_0}(\|\hat P_D - P_{D,f_0}\|_{L^2 \to L^2} \ge  M \delta_N) \le e^{-c N\delta_N^2}$$
and likewise, by the triangle inequality,
\begin{align*}
\mathbb E_f(1-\Psi_N) &= \mathbb P_{f}(\|\hat P_D - P_{D,f_0}\|_{L^2 \to L^2} <  M \delta_N)\\
&=\mathbb P_f(\|\hat P_D - P_{D,f}\|_{L^2 \to L^2} > \|P_{D,f_0}-P_{D,f}\|_{L^2 \to L^2} - M \delta_N) \le e^{-c N\delta_N^2}
\end{align*}
whenever $\|P_{D,f_0}-P_{D,f}\|_{L^2 \to L^2} \ge \bar \delta_N \ge 2M \delta_N$. Now we can apply Theorem \ref{notagain} and deduce that for all $b$ we can choose $M$ and $U$ large enough such that
$$\Pi\big(f: \|f\|_{H^s} \le U, \|P_{D, f}-P_{D,f_0}\|_{L^2 \to L^2} \ge 2 M \delta_N|X_0, \dots, X_{ND}\big) = 1 - O_{\mathbb P_{f_0}}(e^{-bN \delta_N^2}) .$$ This proves Theorem \ref{maint}. To proceed, note that the same arguments work for $\|\cdot\|_{H^2 \to H^2}$ operator norms by appealing to Corollary \ref{opconc2} with the same choice of $J$, resulting in the slower convergence rate $\tilde \delta_N = N^{-(s-1)/(2s+2+d)}$ replacing $\delta_N$. Now to prove Theorem \ref{main} under hypothesis (\ref{sunnyside}), we can invoke the stability estimate Theorem \ref{stabop} and the set inclusion
\begin{align*}&\Big\{f: \|f - f_0\|_{L^2} \le  M\tilde \delta_N\Big\}  \supset \Big\{f:\|f\|_{H^s} \le U, \|P_{D,f} -P_{D,f_0}\|_{H^2 \to H^2} \le  2M \tilde \delta_N \big\}
\end{align*} 
for $M$ large enough such that $M \ge \bar C$. If (\ref{sunnyside}) does not hold we can still use the stability estimate (\ref{logstab0}) from Theorem \ref{logstabthm} and obtain the slower  rate for the posterior distribution. This completes the proof of the contraction rate bounds for $f_0$ in Theorem \ref{main}. The rate for the HS-norms follow in a similar way from (\ref{backlip}), (\ref{backlog}) instead of the previous stability estimates.

\subsubsection{Posterior mean convergence and proof of Theorem \ref{showoff}}\label{show}

The above contraction results holds as well for the `linear' parameter $\theta - \theta_0$, as $\log$ is $L^2$-Lipschitz on $\|\cdot\|_{H^s}$-bounded sets of $f$'s bounded away from zero (and using that $\|f-f_0\|_\infty \to 0$ for $f\to f_0$ in $L^2$ bounded in $H^s$). In turn we further deduce a convergence rate for the posterior mean vectors
\begin{equation}\label{meanrat}
\|E^\Pi[\theta|X_0, \dots, X_{ND}] - \theta_0\|_{L^2} = O_{\mathbb P_{f_0}}(\tilde \delta_N)
\end{equation}
using that we have exponential convergence to zero in (\ref{contract}) for any $b>0$ if we just increase the constant $M$, and by a uniform integrability argument as in Theorem 2.3.2 of \cite{N23} (or see also the proof of Theorem 3.2 in \cite{MNP21}, to whom this argument is due). This then implies the same $L^2(\mathcal O)$-rates for $\bar f_N = f_{E^\Pi[\theta|X_0, \dots, X_{ND}]}$ towards $f_0$ and and in particular implies the second limit in Theorem \ref{showoff}. An argument parallel to the one leading to (\ref{meanrat}) further implies that $\|E^\Pi[\theta|X_0, \dots, X_{ND}]\|_{H^s}=O_{\mathbb P_{f_0}}(1)$ and we can then use (\ref{weakball}) and the imbedding $L^2 \subset (H^1_c)^*$ to obtain convergence to zero of the Hilbert-Schmidt norms $\|P_{D, \bar f_N} -P_{D, f_0}\|_{HS}$ (which bound $\|\cdot\|_{L^2\to L^2}$ norms) also at rate $\tilde \delta_N$.

 \subsection{Neumann eigenfunctions on cylindrical domains } \label{cylinder}

\subsubsection{Proof of Proposition \ref{cyllap}}

Let us decompose a point $x \in \mathcal O_1 \times (0,w)$ as $y=(x_{1}, \dots, x_{d-1})$, $z=x_d.$ The restricted Neumann Laplacians $\Delta_{\mathcal O_1}, \Delta_{(0,w)}$  have  discrete non-positive spectrum on $L^2(\mathcal O_1)$ and $L^2((0,w))$, respectively, with eigenfunctions $e_{1,k}, e_{2, k}, k \in \mathbb N,$ all orthogonal on constants on their respective domains. If we set $e_{1,0}=1/(vol (\mathcal O_1))^{1/2}$, $e_{2,0} = 1/\sqrt w$ for eigenvalues $\lambda_{i, 0}=0 \le  \lambda_{i, k}$ then the eigenfunctions $(e_j: j \ge 0)$ of $\Delta$ on $L^2(\mathcal O)$ tensorise by a standard separation of variables argument (that is left to the reader).
\begin{proposition}
The eigenfunctions of $-\Delta$ on $\mathcal O$ for eigenvalues $\lambda_j = \lambda_{1, k} + \lambda_{2, l}$ are
\begin{equation}
e_{j}(y,z) = e_{1, k}(y) \times e_{2, l}(z),~~~j= (k, l) \in \mathbb N^2 \cup \{0,0\}, ~y \in \mathcal O_1, z \in (0,w).
\end{equation}
\end{proposition} 

To proceed, recall that for a convex domain $\mathcal O_1$, the Poincar\'e constant satisfies $p(\mathcal O_1)\le (diam(\mathcal O_1)/\pi)^2$ by a classical result of \cite{PW60}. For simple eigenvalues we then have:

\begin{proposition}\label{cyllapsimp}
Suppose that the Poincar\'e constant $p(\mathcal O_1)$ of $\mathcal O_1$ satisfies $p(\mathcal O_1) \le w^2/2\pi^2.$ Then the first non-zero eigenvalue $\lambda_1$ of $\Delta$ on $\mathcal O= \mathcal O_1 \times (0,w)$ is simple, equals $\pi^2/w^2$ and the rest of the spectrum is separated from $\lambda_1$ by at least $\pi^2/w^2$. The corresponding eigenfunction is smooth in the strict interior of $\mathcal O$ and satisfies for all $\eta>0$ small enough
\begin{equation}\label{gradual}
\inf_{x: |x-\partial \mathcal O|_{\R^d} \ge \eta} |\nabla e_1 (x)|_{\R^d} \ge \frac{\pi^2 \eta}{2 w^2 \sqrt{vol(\mathcal O_1)}} > 0.
\end{equation}
\end{proposition}
\begin{proof}
By the assumption and (\ref{topevar}) we have $\lambda_{1,1} \ge 1/p(\mathcal O_1)$. The first eigenvalue $\lambda_{2,1}$ of $\Delta$ on $(0,w)$ is $\pi^2/w^2$, hence $\lambda_{2,1}<\lambda_{1,1}$ and the first non-constant eigenfunction of $\Delta$ on $\mathcal O$ corresponds to $\lambda_1 = 0+ \lambda_{2,1}$ and equals 
\begin{equation}\label{theone}
e_{1}(y,z) = e_{1,0}(y) e_{2,1}(z)= \frac{\cos(\pi z/w)}{\sqrt w \sqrt{vol(\mathcal O_1)}},~~y \in \mathcal O_1, 0<z<w.
\end{equation}
By the hypotheses the next eigenvalue satisfies $\lambda_{2} \ge \min \big(1/p(\mathcal O_1), \frac{4\pi^2}{w^2} \big)$ and so we have a `two-sided' spectral gap around $\lambda_1$ in the spectrum $\sigma(\Delta_{\mathcal O})$ in the sense that 
\begin{equation}
\sigma(\Delta_{\mathcal O}) \cap (\lambda_1 - \epsilon, \lambda _1 + \epsilon) = \{\lambda_1\}~~\text{ for } \epsilon = \min \Big(\frac{\pi^2}{w^2},  \frac{1}{p(\mathcal O_1)} - \frac{\pi^2}{w^2} \Big).
\end{equation}
By the assumption on $p(\mathcal O_1)$ the first claim follows. Next for $x=(y,z)$ away from the boundary we have $\min(z, 1-z) \ge \eta>0$ and so we have
\begin{align*}
|\nabla e_1 (x)|^2_{\R^d} & = \sum_{j=1}^d \big[\frac{\partial e_1(x) }{\partial x_j}\big]^2 = \frac{1}{w}\frac{1}{vol (\mathcal O_1)} \Big(\frac{d}{dz} \cos\big(\frac{\pi z}{w}\big) \Big)^2 \\
& = \frac{\pi^2}{w^3 vol (\mathcal O_1)} \Big( \sin\big(\frac{\pi z}{w}\big) \Big)^2 \ge \frac{\pi^4 \eta^2}{4 w^5 vol(\mathcal O_1)} > 0,
\end{align*}
for $\eta$ small w.l.o.g. (so that we can use $\sin u \ge u/2$ for $u$ near zero).
\end{proof}

If in the previous proof we only assume $p(\mathcal O_1) \le \frac{w^2}{\pi^2}$ then the first eigenvalue of $\Delta_{\mathcal O_1}$ may co-incide with the one of $(0,w)$ and there may then be multiple eigenfunctions for $\lambda_{1}$. But the eigenfunction (\ref{theone}) is still one permissible choice, and we can choose the weight $\iota$ in (\ref{eigenblock}) to choose that eigenfunction, so that Proposition \ref{cyllap} remains valid also in this case.

\subsubsection{Proof of Theorem \ref{cylinderth}, Step I: perturbation}

The remainder of this section is devoted to the proof of Theorem \ref{cylinderth}. It consists of combining Proposition \ref{cyllap} with perturbation theory for linear operators. The following result will be used repeatedly. For a proof see Sec.s~IV.3.4-5 in Kato \cite{K76} (or cf.~also Proposition 4.2 in \cite{GHR04}). The clusters of the eigenvalues converge also without simplicity of $\lambda_{1,f_0}$, see the discussion in \cite{K76} or also in Sec 2.3 in \cite{KLN20}.
\begin{proposition}\label{opperture}
Let $K$ be a bounded linear self-adjoint operator on a separable Hilbert space $H$ with discrete spectrum $\sigma(K)$ and simple eigenvalue $\kappa$ such that $\sigma(K) \cap [\kappa - \epsilon, \kappa + \epsilon] = \{\kappa\}$ for some $\epsilon>0$. Let $K_\delta$ be another self-adjoint linear operator such that $\|K-K_\delta\|_{H \to H} < \epsilon/4$. Then $K_\delta$ has a simple eigenvalue $\kappa_\delta \in (\kappa - \epsilon/2, \kappa + \epsilon/2)$ and there are eigenvectors $k, k_{\delta}$ of $K, K_\delta$ for $\kappa, \kappa_\delta$ such that $\|k-k_\delta\|_H\to 0$ as $\epsilon \to 0$.
\end{proposition}

\subsubsection{Step II: rounding the corners}

Let us fix $w \ge 2$ and agree to write $\mathcal O_m \equiv \mathcal O_{m,w}, m \in \mathbb N,$ for the sequence of domains from (\ref{om}), as well as $\mathcal O = \mathcal O_{(w)}$ for the limit set, in this subsection. Note that  $\mathcal O_1$ is the largest domain containing all the others and the perturbation argument below will be given on the Hilbert space $L^2(\mathcal O_1) \supset L^2(\mathcal O_m) \supset L^2(\mathcal O)$, where the inclusions are to be understood by restriction to, and zero extension from, the domains $\mathcal O, \mathcal O_m$. [Note a slight abuse of notation that $\mathcal O_1$ is not the cylinder base from earlier.]

Consider the linear operators on $L^2(\mathcal O_m)$ given by $T_{1,1} = (id + \Delta_{\mathcal O_m})^{-1}$ from after (\ref{h1est}) in Section \ref{specsec} with $f=1, \mathcal O=\mathcal O_m$. We extend them to operators denoted by $T_{\mathcal O_m}$ on $L^2(\mathcal O_1)$ by restriction of $h \in L^2(\mathcal O_1)$ to $\mathcal O_m$ and zero-extension of the resulting functions $T_{1,1}(h)$ outside of $\mathcal O_m$. Likewise we define $T_{\mathcal O}$ on $L^2(\mathcal O_1)$.

\begin{lemma}
We have as $m \to \infty $ that $\|T_{\mathcal O_m} - T_{\mathcal O}\|_{L^2(\mathcal O_1) \to L^2(\mathcal O_1)} \to 0.$
\end{lemma}
\begin{proof}
For any $h$ such that $\|h\|_{L^2(\mathcal O_m)} \le \|h\|_{L^2(\mathcal O_1)} \le 1$ and writing $u_m(h) = T_{\mathcal O_m}(h)$, we have from Theorem 3.1.3.3 in \cite{G11} (with $\lambda=1$ there) that 
\begin{equation}\label{keyregbd}
\|u_m(h)\|_{H^2(\mathcal O_m)} \le C \|h\|_{L^2(\mathcal O_m)} \le C,
\end{equation}
where $C$ is a numerical constant independent of $\mathcal O_m, h$. Following the argument given after (3.2.1.8) in \cite{G11} one shows that $u_m(h) \to u(h)=T_{\mathcal O}(h)$ weakly in $H^2(\mathcal O)$ and then by compactness also in the norm of $L^2(\mathcal O)$ and in fact of $L^2(\mathcal O_1)$ for the given $h$. This convergence is uniform in $h$: indeed, suppose $u_m(h)$ does not converge to $u(h)$ in $L^2(\mathcal O_1)$ uniformly in $\|h\|_{L^2(\mathcal O_1)} \le 1$. Then there exists $\epsilon>0$ and a sequence $h_m \in L^2(\mathcal O_1)$ such that $\|h_m\|_{L^2(\mathcal O_1)} \le 1$ for which
\begin{equation}\label{tocont}
\|u_m(h_m) - u(h_m)\|_{L^2(\mathcal O_1)} \ge \epsilon_0>0~~ \textit{for all }m.
\end{equation}
The sequence $h_{m}$ converges in the dual space $(H^1(\mathcal O_1))^*$ to some $h$ along a subsequence, by compactness of the inclusion $L^2 \subset (H^1)^*$. As $T_{\mathcal O_m}$ is self-adjoint on $L^2(\mathcal O_m)$ we deduce
\begin{align*}
\|u_m(h_m)-u_m(h)\|_{L^2} &= \sup_{\|\psi\|_{L^2(\mathcal O_m)} \le 1}\big|\langle T_{\mathcal O_m}\psi, h_m-h \rangle_{L^2(\mathcal O_m)}  \big| \\
&\le \|h_m - h\|_{(H^1(\mathcal O_m))^*} \sup_{\|\psi\|_{L^2(\mathcal O_m)} \le 1} \|T_{\mathcal O_m}(\psi)\|_{H^1(\mathcal O_m)} \\
&\lesssim \|h_m-h\|_{(H^{1}(\mathcal O_1))^*}  \to_{m \to \infty} 0 
\end{align*}
using also that the restriction operator from $\mathcal O_1$ to $\mathcal O_m$ is continuous from $(H^1(\mathcal O_1))^*$ to $(H^1(\mathcal O_m))^*$, and where the last supremum was bounded using (\ref{h1est}) (with $f=1$) and the Cauchy-Schwarz inequality, by $\sup_{\|\psi\|_{L^2(\mathcal O_m)} \le 1}\|T_{\mathcal O_m}(\psi)\|^{1/2}_{L^2(\mathcal O_m)} \le 1$, 
since $T_{\mathcal O_m}$ has $L^2 \to L^2$ norm at most one as its eigenvalues satisfy $1/(1+\lambda_{j,m}) \le 1$ for all $j \ge 0, m$. The same argument implies that $u(h_m) \to u(h)$ in in $L^2(\mathcal O_1)$. From what precedes we deduce that
$$\|u_m(h_m)-u(h_m)\|_{L^2} \le \|u_m(h_m)-u_m(h)\|_{L^2} + \|u_m(h)-u(h)\|_{L^2} + \|u(h)-u(h_m)\|_{L^2}$$ converges to zero as $m \to \infty$, which contradicts (\ref{tocont}), and proves the lemma.
\end{proof}

Just as after (\ref{h1est}), the eigenvalues of the limiting operator $T_{\mathcal O}$ are $1, (1+\lambda_{1,1})^{-1}, (1+\lambda_{2,1})^{-1}, \dots,$ for eigenfunctions $1_{\mathcal O}, e_{1,1}, e_{2,1}, \dots$ of $\Delta_\mathcal O$ extended by zero outside of $\mathcal O$. [Note that $L^2(\mathcal O_1) = L^2(\mathcal O) \oplus L^2(\mathcal O_1 \setminus \mathcal O)$ is an orthogonal sum.] By Proposition \ref{cyllapsimp}, the eigenvalue $(1+\lambda_{1,1})^{-1}$ is isolated and simple. Similarly, the eigenpairs of $T_{\mathcal O_m}$ are $((1+\lambda_{j,1,m})^{-1}, e_{j,1,m})$ with eigenfunctions extended by zero outside of $\mathcal O_m$, and from Proposition \ref{opperture}  we deduce that $(\lambda_{1,1,m}, e_{1,1,m}) \to (\lambda_{1,1}, e_{1,1})$ as $m \to \infty$ in $\R \times L^2(\mathcal O_1)$. Moreover in any strict interior subset of $\mathcal O$ containing $\mathcal O_0$, the eigenfunctions $e_{1,1,m}, e_{1,1}$ have uniformly bounded Sobolev norms of any order (e.g., use \cite{E10}, p.334, Thm 2) and so by a standard compactness argument for Sobolev norms and the Sobolev imbedding $H^{\alpha} \subset C^2, \alpha>2+d/2$, we obtain convergence of 
\begin{equation}\label{c2ef}
e_{1,1,m} \to e_{1,1} \text{ in } C^2(\mathcal O_0).
\end{equation}
Thus the gradient condition (\ref{gradual}) for $e_{1,1}$ is inherited by $e_{1,1,m}$ for all $m$ large enough depending on the lower bound in (\ref{gradual}). Also $|\Delta e_{1,1,m}|$ remains bounded on $\mathcal O_0$ by a fixed constant in view of (\ref{c2ef}), so we can verify (\ref{sunnyside}) for $\mu$ large enough and some $c_0>0$. This completes the proof of Theorem \ref{cylinderth}A).

\subsubsection{Step III: neighbourhood of $\Delta$}

We now extend the previous result to a neighbourhood of $f=1$. As the domain is fixed in what follows, we just write $\mathcal O$ for the bounded convex smooth domain $\mathcal O_{m,w}=\mathcal O_m$ from the previous subsection.

\begin{lemma}\label{lappert}
Regarding $\mathcal L_{f}^{-1}, \mathcal L_{1}^{-1}$  as bounded linear operators on $L^2_0(\mathcal O)$ we have for some $D'=D'(f_{min}, \|f\|_\infty, \mathcal O)$ that
$\|\mathcal L_{f}^{-1}- \mathcal L_{1}^{-1}\|_{L^2_0 \to L^2_0} \leq D' \|f-1\|_\infty.$
\end{lemma}
\begin{proof}
For $\phi \in L^2_0$ denote by $u_f= \mathcal L_f^{-1}(\phi)$ the solution to (\ref{aeneumann}). By Proposition \ref{sobald} we have $\bar H^1_1 \subset H^1$ and so since $\mathcal L_1^{-1}$ is self-adjoint and using the divergence theorem,
\begin{align*}
&\|\mathcal L^{-1}_{f}\phi- \mathcal L_{1}^{-1} \phi\|_{L^2} \lesssim \|\mathcal L_{1}^{-1}[\nabla \cdot (1-f) \nabla u_f]\|_{L^2}  = \sup_{\|\varphi\|_{L^2} \le 1, \int \varphi =0} \Big|\int_\mathcal O \nabla \cdot (1-f) \nabla u_f \mathcal L_1^{-1}[\varphi] \Big| \\
& \lesssim \sup_{\|\psi\|_{H^1} \le 1}  \left|\int_\mathcal O (f-1) \nabla \psi \cdot \nabla u_f \right|  \lesssim \|f-1\|_\infty \sup_{\|\psi\|_{H^1} \le 1}\|\psi\|_{H^1} \|\nabla u_f\|_{L^2} \lesssim \|\phi\|_{L^2} \|f-1\|_\infty
\end{align*}
where we use $\| u_f\|_{H^1} \lesssim \|\phi\|_{L^2}$ as follows from the results in Section \ref{specsec}.
\end{proof}

By the arguments after (\ref{c2ef}), (\ref{h1est}), the operator $-\mathcal L_{1}^{-1}$ has a simple eigenvalue $\lambda_{1,1}$ with eigenfunction $e_{1,1}$ satisfying (\ref{sunnyside}). We apply the preceding lemma and Proposition \ref{opperture} in the Hilbert space $L^2_0(\mathcal O)$, which implies the convergence of the eigenpair $(\lambda_{1,f}, e_{1,f})$ of $-\mathcal L_f^{-1}$ to $(\lambda_{1,1}, e_{1,1})$ as $\|f-1\|_\infty \to 0$, in $\R \times L^2(\mathcal O)$. Under the hypotheses on $f$, Theorem 2 on p.334 in \cite{E10} implies that the $\|e_{1,f}\|_{H^k(V)}$ norms in a strict interior subset $V \supset\mathcal O_0$ of $\mathcal O$ are all uniformly bounded for $k>2+d/2$.  The standard interpolation inequality for Sobolev norms (p.44 in \cite{LM72}) implies for some $0<c(k,\alpha)<1$, and $2+d/2<\alpha<k$ (if necessary considering fractional Sobolev norms)
\begin{equation} \label{intpolef}
\|e_{1,f} - e_{1,1}\|_{H^\alpha} \le \|e_{1,f} - e_{1,1}\|^{c(k, \alpha)}_{L^2} \|e_{1,f} - e_{1,1}\|^{1-c(k, \alpha)}_{H^k} \to 0
\end{equation}
as $\|f-1\|_\infty \to 0$, where all Sobolev norms are over $V$. Since $H^\alpha$ embeds continuous into $C^2$ this implies convergence to zero of $\|e_{j,f} - e_{j,1}\|_{C^2(V)}$. We can then verify (\ref{sunnyside}) just as after (\ref{c2ef}), for $\kappa$ small enough, completing the proof of Theorem \ref{cylinderth}.

\subsection{Proofs of auxiliary results}\label{aux}

\subsubsection{Proof of Proposition \ref{sobald}}

We require a few preparatory remarks that will be used: For any $\eta>0$ the Sobolev imbedding gives $\|f\|_\infty \le \|f\|_{C^1} \lesssim \|f\|_{H^{1+d/2+\eta}}\le U.$ The multiplier inequality 
\begin{equation}\label{multi}
\|f h\|_{H^r} \lesssim \|f\|_{B^r} \|h\|_{H^r} \le U \|h\|_{H^r},~~ r \le s,
\end{equation}
where $B^r = H^r$ for $r > d/2$ and $B^r=C^r$ for $r \le d/2$, is also standard, and where we use that $H^s$ imbeds continuously into $C^{s-d/2-\eta}\subset C^r$ for $r \le d/2$ in case B) of the proposition. We also recall the standard result from elliptic PDEs that $(\Delta, \partial/\partial \nu)$ is a continuous isomorphism between $H^{k}(\mathcal O) \cap L^2_0(\mathcal O)$ and $H^{k-2}(\mathcal O) \cap L^2_0 \times H^{k-3/2}(\partial \mathcal O) $ (e.g, Theorem II.5.4 in \cite{LM72} or Theorem 4.3.3 in \cite{T83}), specifically
\begin{equation}\label{isom}
\|u\|_{H^k} \simeq \|\Delta u\|_{H^{k-2}} + \|\partial u/\partial \nu\|_{H^{k-3/2}},~~u \in H^k,~k \ge 2,
\end{equation}
with constants depending only on $d, \mathcal O, k$. [Here the $H^\alpha$-spaces on the boundary $\partial \mathcal O$ are naturally defined as in \cite{LM72}, and we note that the result is also true when $d=1$ if we replace the boundary spaces simply by the values of $u'$ at the endpoints of the interval $\mathcal O$.]

Now any $\varphi \in \bar H^k_f$ is the limit in $\bar H^k_f$ and in $L^2$ of its partial sum $\varphi_J = \sum_{j \le J} e_{j,f} \langle \varphi, e_{j,f} \rangle_{L^2}$. Moreover the $\varphi_J$ lie in $H^1_\nu \cap \bar H^k_f$ since the $e_j$'s do. We then have from (\ref{greenid}) and for constants in $\simeq$ depending only on $f_{min}, U \ge \|f\|_\infty$, the two-sided inequality
\begin{equation}\label{h1eq}
\|\varphi_J\|^2_{H^1} = \|\nabla \varphi_J\|_{L^2}^2 \simeq  \|\sqrt f \nabla \varphi_J\|_{L^2}^2 = \langle \mathcal L_f \varphi_J, \varphi_J \rangle_{L^2} = \|\varphi_J\|_{\bar H^1_f}^2.
\end{equation}
Taking limits, these inequalities extend to all $\varphi \in \bar H^1_f$, in particular $\bar H^1_f \subset H^1$. The inclusion $H^1  \subset \bar H^1_f$ is also valid (p.474 in \cite{TI}, or see Exercise 38.1 in \cite{B11}) but will be left to the reader. This proves the required assertions when $k=1$.

For $k=2$, using (\ref{isom}), (\ref{h1eq}), $\phi_J \in H^1_\nu$, we have with constants depending on $U,f_{min}$,
\begin{align*}
\|\varphi_J\|_{H^2} &\lesssim \|\Delta \varphi_J \|_{L^2} = \|f^{-1}(\mathcal L_f \varphi_J-\nabla f \cdot \nabla \varphi_J)\|_{L^2}  \lesssim \|\mathcal L_f \varphi_J\|_{L^2} + \|f\|_{C^{1}}\|\varphi_J\|_{H^{1}}  \lesssim \|\varphi_J\|_{\bar H^2_f}
\end{align*} and again taking limits the result extends to all $\phi \in \bar H^2_f$, in particular $\bar H^2_f \subset H^2$. We see that any $\phi \in \bar H^k_f, k \ge 2,$ is the $H^2$-limit of elements in $H^2$ satisfying Neumann boundary conditions. From this and Theorem I.9.4 in \cite{LM72} we deduce that $\bar H^2_f \subset H^2 \cap H^1_\nu.$ Then for $h \in H^2 \cap H^1_\nu \cap L^2_0$ and $f \in C^1$ we have $\|\mathcal L_fh\|_{L^2} \leq C(U) \|h\|_{H^2}<\infty$ and by the spectral representations of $\mathcal L_f, h \in L^2_0,$ we deduce $\mathcal L_f h \in L^2_0$. The inclusion of the r.h.s.~in (\ref{pain}) into $\bar H^2_f$ is also clear since for such $\varphi$ we have from the divergence and Parseval's theorem $$\|\varphi\|^2_{\bar H^2} = \sum_{j \ge 1} \lambda^2_{j,f} \langle \varphi, e_{j,f} \rangle_{L^2}^2 = \sum_{j \ge 1} \langle \mathcal L_f \varphi, e_{j,f} \rangle_{L^2}^2 = \|\mathcal L_f \varphi\|_{L^2}^2<\infty,$$ so that combining what precedes, (\ref{pain}) is proved. The desired norm equivalence for $k=2$ then also follows from the last estimates.

The claims for integer $k >2$ follow by induction. We assume the result has been proved for $k-1$ and $k-2$. Then we have $\bar H^k_f \subset H^1_\nu \cap H^{k-2}$. We then see from (\ref{isom}) that on $\bar H^k_f$, the norms $\|\cdot\|_{H^k}$ are equivalent to the norms $\|\Delta(\cdot)\|_{H^{k-2}}$. In particular for $\varphi \in \bar H^k_f$,
\begin{align*}
\|\varphi\|_{H^k} &\lesssim \|\Delta \varphi\|_{H^{k-2}} = \|f^{-1}(\mathcal L_f \varphi-\nabla f \cdot \nabla \varphi)\|_{H^{k-2}}\\
& \lesssim \|\mathcal L_f \varphi\|_{H^{k-2}} + \|f\|_{B^{k-1}} \|\varphi\|_{H^{k-1}}  \lesssim \|\mathcal L_f \varphi\|_{\bar H^{k-2}_f} + \|\varphi\|_{\bar H^{k-1}_f} \lesssim \|\varphi\|_{\bar H^k_f}
\end{align*}
using also the induction hypothesis, the multiplier inequality, and the definition of $\bar H^k_f$. The preceding bound for $\|\Delta \varphi\|_{H^{k-2}}$ in particular implies $\varphi \in H^k$. In the other direction, by similar arguments, 
\begin{align*}
\|\varphi\|_{\bar H^k_f} &= \|\mathcal L_f \varphi\|_{\bar H^{k-2}_f}  \lesssim  \|\Delta \varphi\|_{\bar H^{k-2}_f} + \|f\|_{B^{k-1}} \| \varphi\|_{\bar H^{k-1}_f} \lesssim \|\Delta \varphi \|_{H^{k-2}} + \|\varphi\|_{H^{k-1}}  \lesssim \|\varphi\|_{H^k}.
\end{align*}

The last assertions follow for $k=1$ from $H^1 = \bar H^1_f=\bar H^1_{f'}$ and for $k=2$ from (\ref{pain}). The general case follows again by induction: indeed suppose the result holds for some $k$. Just as when showing (\ref{pain}), the space $\bar H^{k+2}_f$ consists precisely of all $\phi \in \bar H_f^k$ satisfying Neumann boundary conditions and such that $\mathcal L_f \phi \in \bar H^k_f$. This immediately implies $H^k_c/\R \cap L^2_0 \subset \bar H^k_1$ as elements of $H^k_c/\R \cap L^2_0$ are of the form $\bar \varphi = \varphi - \int \varphi$ for some $\varphi \in H^k_c$ so its normal derivatives of all orders vanish at $\partial \mathcal O$, and $\mathcal L_f \bar \varphi \in H^{k-2}_c \subset \bar H^{k-2}_f$ by the induction hypothesis. Finally, since $\bar H^k_f=\bar H^k_{f'}$ by the induction hypothesis, we have $\mathcal L_f \phi \in \bar H^k_{f'}$ and so $\phi \in \bar H^k_{f'}$. The equivalence of norms then follows from the first part of the proposition.

\subsubsection{Proof of Proposition \ref{translow}}

We will apply Theorem 3.1 in \cite{C03} with semi-group $e^{-t E_f}$ acting on $L^2(\mathcal O)$, where $E_f$ is the closure of $-\mathcal L_f$ from before (\ref{diric}) on the domain $H^1$. We note that any bounded convex domain satisfies the `chain condition' employed in that reference. Further, the doubling condition (D) there is  satisfied with scaling constant $\nu=d$. The upper bound heat kernel estimate for $p_t$ required in (3.1) in Theorem 3.1 in \cite{C03} is proved in Theorem 3.2.9 in \cite{D89} for the value $w=2$ (noting that a bounded domain with smooth boundary satisfies the `extension property' for Sobolev spaces required in \cite{D89}). Finally
\begin{equation} \label{LI}
\sup_{x,y \in \mathcal O}\frac{|\varphi(x)-\varphi(y)|}{|x-y|^{\alpha-d/2}} \lesssim \|\mathcal L_f^{\alpha/2} \varphi \|_{L^2} ~~\forall \varphi \in \bar H^\alpha_f,~\alpha>d/2
\end{equation}
where $\mathcal L_f^{\alpha/2}$ is the $\alpha/2$-fold application of $\mathcal L_f$. This verifies Condition (3.2) in \cite{C03} (for the choice of $\varphi=p_{t,f}$ relevant in the proof of Theorem 3.1 there). To prove (\ref{LI}), the Sobolev imbedding $H^\alpha \subset C^{\alpha-(d/2)}$ and Proposition \ref{sobald} imply that it suffices to bound $\|\varphi\|_{\bar H_f^{\alpha}}$, which for $\varphi \in \bar H^\alpha_f$ equals the graph norm $\|\mathcal L_f^{\alpha/2} \varphi \|_{2}$ by the argument given in the last paragraph of the proof of Proposition \ref{sobald}. This completes the proof.

 \bibliography{diffusion_bvm}{}
\bibliographystyle{imsart-number}

\end{document}